\documentclass[11pt,reqno]{amsart}
\usepackage{amsmath,amsthm,amsfonts,amssymb,mathrsfs,bm,graphicx,stmaryrd}
\usepackage{mathtools}
\usepackage[usenames]{color}
\usepackage{multicol}
\usepackage[colorlinks=true,linkcolor=blue]{hyperref}
\usepackage[letterpaper,hmargin=1.0in,vmargin=1.0in]{geometry}
\parskip 	\smallskipamount

\newtheorem{theorem}{Theorem}[section]
\newtheorem{lemma}[theorem]{Lemma}

\newtheorem{proposition}[theorem]{Proposition}
\newtheorem{observation}[theorem]{Observation}
\newtheorem{remark}[theorem]{Remark}
\newtheorem{maintheorem}{Theorem}
\def\N{\mathbb{N}}
\def\L{\mathbb{L}}
\def\P{\mathbb{P}}
\def\Z{\mathbb{Z}}
\def\R{\mathbb{R}}
\def\E{\mathbb{E}}
\def\Var{{\rm Var}}
\newcommand{\cf}{\mathcal{F}}
\newcommand{\cA}{\mathcal{A}}
\newcommand{\cB}{\mathcal{B}}
\newcommand{\cC}{\mathcal{C}}
\newcommand{\cD}{\mathcal{D}}
\newcommand{\cE}{\mathcal{E}}

\newcommand{\sL}{\mathcal{L}}
\newcommand{\ce}{\mathcal{E}}

\newcommand{\cg}{\mathcal{G}}
\newcommand{\sS}{\mathcal{S}}
\newcommand{\sT}{\mathscr{T}}
\newcommand{\ch}{\mathcal{H}}

\newcommand{\red}{\textcolor{red}}
\newcommand{\blue}{\textcolor{blue}}
\begin{document}
\title[Time Correlations in LPP]{Time Correlation Exponents in Last Passage Percolation}

\author{Riddhipratim Basu}
\address{Riddhipratim Basu, International Centre for Theoretical Sciences, Tata Institute of Fundamental Research, Bangalore, India}
\email{rbasu@icts.res.in}
\author{Shirshendu Ganguly}
\address{Shirshendu Ganguly, Department of Statistics, UC Berkeley, Berkeley, CA, USA}
\email{sganguly@berkeley.edu}

\date{}
\maketitle

\begin{abstract}
For directed last passage percolation on $\Z^2$ with exponential passage times on the vertices, let $T_{n}$ denote the last passage time from $(0,0)$ to $(n,n)$. We consider asymptotic two point correlation functions of the sequence $T_{n}$. In particular we consider ${\rm Corr}(T_{n}, T_{r})$ for $r\le n$ where $r,n\to \infty$ with $r\ll n$ or $n-r \ll n$. We show that in the former case ${\rm Corr}(T_{n}, T_{r})=\Theta((\frac{r}{n})^{1/3})$ whereas in the latter case $1-{\rm Corr}(T_{n}, T_{r})=\Theta ((\frac{n-r}{n})^{2/3})$. The argument revolves around finer understanding of polymer geometry and is expected to go through for a larger class of integrable models of last passage percolation. As by-products of the proof, we also get a couple of other results of independent interest: Quantitative estimates for locally Brownian nature of pre-limits of Airy$_2$ process coming from exponential LPP, and precise variance estimates for lengths of polymers constrained to be inside thin rectangles at the transversal fluctuation scale. 
\end{abstract}
\tableofcontents
\section{Introduction and statement of results}\label{first}
We consider directed last passage percolation on $\Z^2$ with i.i.d.\ exponential weights on the vertices. We have a random field $$\omega=\{\omega_{v}: v\in \Z^2\}$$ where $\omega_v$ are i.i.d standard Exponential variables. For any two points $u$ and $v$ with $u\preceq v$ in the usual partial order, we shall denote by $T_{u,v}$ the last passage time from $u$ to $v$; i.e., the maximum weight among all weights of all directed paths from $u$ to $v$ (the weight of a path is the sum of the field {along the path}). By $\Gamma_{u,v}$, we shall denote the almost surely unique path that attains the maximum- this will be called a polymer or a geodesic. This is one of the canonical examples of an integrable model in the so-called KPZ universality class \cite{Jo99, BF08}, and has been extensively studied also due to its connection to Totally Asymmetric Simple Exclusion process in $\Z$. For notational convenience let us denote $(r,r)$ for any $r\in \N$ by ${\bf{r}}$ and $T_{\mathbf{0}, \mathbf{n}}$ by $T_{n}$ and similarly $\Gamma_{\mathbf{0}, \mathbf{n}}$ by $\Gamma_{n}$. It is well known \cite{Jo99} that $n^{-1/3}(T_n-4n)$ has a distributional limit (a scalar multiple of the GUE Tracy-Widom distribution), and further it has uniform (in $n$) exponential tail estimates \cite{BFP12}. 
Although by now the scaled and centered field obtained from $\{T_{{\bf{0}},(x,y)}\}_{x+y=2n}$ using the KPZ scaling factors of $n^{2/3}$ in space and $n^{1/3}$ in polymer weight has been intensively studied and the scaling limit as $n\to \infty$ identified to be the Airy$_2$ process, much less is known about the evolution of the random field in time i.e., across  various values of $n$, see \cite{BL17, MQR17} for some recent progress. 

In this paper, we study two point functions describing the `aging' properties of the above evolution. More precisely we investigate the correlation structure of the  the tight sequence of random variables  $n^{-1/3}(T_n-4n)$ across $n$. In particular, let us define for $r\le n\in \N$
$$ \rho(n,r)=: \mbox{Corr} (T_{n}, T_{r} ).$$
We are interested in the dependence of  $\rho(n,r)$ on $n$ and $r$ as they become large. Observe that the FKG inequality implies that $\rho(n,r) \geq 0$. Heuristically, one would expect that $\rho(n,r)$ is close to $1$ for if $|n-r|\ll n$ and close to $0$ for $r\ll n$. 

Our main result in this paper establishes the exponents governing the rate of correlation decay and thus identifies up to constants the asymptotics of $\rho$ in these regimes. Namely we show $$\rho(n,r)=\Theta((\frac{r}{n})^{1/3}) \text{ if  } 1\ll r\ll n \text{ and } \rho(n,r)=1-\Theta ((\frac{n-r}{n})^{2/3}) \text{ if } 1\ll  n-r \ll n.$$ It turns out that the upper bound in the former case is similar to the lower bound in the latter case, and the lower bound in the former case is similar to to the upper bound in the latter case. We club these statements in the following two theorems.

\begin{maintheorem}
\label{t:small} 
There exists $r_0\in \N$ and positive absolute constants $\delta_1$, $C_1$, $C_2$ such that the following hold.
\begin{enumerate}
\item[(i)] For $r_0<r<\delta_1 n$ and for all $n$ sufficiently large we have 
$$ \rho (n,r)\leq C_1 \left( \frac{r}{n} \right)^{1/3}.$$
\item[(ii)] For $r_0<n-r < \delta_1 n$ and for all $n$ sufficiently large we have 
$$ 1-\rho (n,r)\leq C_2 \left( \frac{n-r}{n} \right)^{2/3}.$$
\end{enumerate}
\end{maintheorem}

\begin{maintheorem}
\label{t:large}
There exists $r_0\in \N$ and positive absolute constants $\delta_1$, $C_3$, $C_4$ such that the following hold.
\begin{enumerate}
\item[(i)] For $r_0<r<\delta_1 n$ and for all $n$ sufficiently large we have 
$$ \rho (n,r)\geq C_3 \left( \frac{r}{n} \right)^{1/3}.$$
\item[(ii)] For $r_0<n-r < \delta_1 n$ and for all $n$ sufficiently large we have 
$$ 1-\rho (n,r)\geq C_4 \left( \frac{n-r}{n} \right)^{2/3}.$$
\end{enumerate}
\end{maintheorem}

Before proceeding further we record a few remarks concerning the history of the problem and related works.

\begin{remark}
\label{r:history}
These exponents were conjectured in \cite{FS16} using partly rigorous analysis, and as far as we are aware was first rigorously obtained in an unpublished work of Corwin and Hammond \cite{CH14+} in the context of Airy line ensemble using the Brownian Gibbs property of the same established in \cite{CH14}. 
\end{remark}

\begin{remark}
\label{r:difference}
Days before posting this paper on arXiv, we came across \cite{FO18} which considers the same problem. Working with rescaled last passage percolation \cite{FO18} analyzes the limiting quantity $r(\tau):=\lim_{n\to\infty} {\rm Corr} (T_{n}, T_{\tau n})$. They establish the existence of the limit and consider the $\tau\to 0$ and $\tau\to 1$ asymptotics establishing the same exponents as in Theorem \ref{t:small} and Theorem \ref{t:large}. The approach  in \cite{FO18} appears to be using comparison with stationary LPP using exit points \cite{CP15} together with using weak convergence to Airy process leading to natural variational formulas. In the limiting regime they get a sharper estimate obtaining an explicit expression of the first order term, providing rigorous proofs of some of the conjectures in \cite{FS16}. In contrast, our approach hinges on using the moderate deviation estimates for point-to-point last passage time to understand local fluctuations in polymer geometry following the approach taken in \cite{BSS14, BSS17B} leading to results for finite $n$ and also allowing us to analyze situations where $r\ll n$ or $n-r \ll n$ which can't be read off from weak convergence. Our work is completely independent of \cite{FO18}.
\end{remark}

In the process of proving Theorems \ref{t:small} and \ref{t:large} we prove certain auxiliary results of independent interest. 
First we establish a local regularity property of the pre-limiting profile of Airy$_2$ process obtained from the exponential LPP model. This result is of independent interest. We need to introduce some notations before making the formal statement. For $n\in \N, s\in \Z$ with $|s|<n$ we define 
$$ L_{n,s}:= T_{\mathbf{0},(n+s,n-s)}.$$
It is known \cite{BF08} that 
$$\mathcal{L}_n(x):=2^{-4/3}n^{-1/3}(L_{n,x(2n)^{2/3}}-4n)$$ converges weakly to the $\mathcal{A}_2(x)-x^{2}$ where $\mathcal{A}_2(\cdot)$ denotes the stationary Airy$_2$ process. It is known that the latter locally looks like Brownian motion \cite{H16,P17}
and hence one would expect that $\mathcal{L}(x)-\mathcal{L}(0)$ will have a fluctuation of order $x^{1/2}$ for small $x$. We prove a quantitative version of the same at all shorter scales.

\begin{maintheorem}
\label{t:aest}
There exist constants $s_0>0,z_0>0$ and $C,c>0$ such that the following holds for all $s>s_0$, $z>z_0$ and for all $n> Cs^{3/2}$: 
$$\P\left (\sup_{s': |s'|<s} L_{n,s'}-L_{n,0}  \geq zs^{1/2}\right) \leq e^{-cz^{4/9}}.$$ 
\end{maintheorem}

\begin{remark}
\label{r:randn}
The estimate for the case $n<Cs^{3/2}$ is relatively easy and follows for e.g.\ using Proposition \ref{t:treesup}. So the most interesting case is when $s \ll n^{2/3}$. Such an estimate in this regime was first obtained in \cite{H16} for Brownian last passage percolation using the Brownian Gibbs resampling property of the pre-limiting line ensemble in that model. Observe that one would expect the Gaussian exponent $z^2$ in the upper bound of the probability in the statement of the theorem and that is what is obtained in \cite{H16}. We, on the other hand, use cruder argument to obtain only a stretched exponential decay. However, the exponent $4/9$ is not optimal even for our arguments, we have not really tried to optimize it. 
\end{remark}

We also prove up to constants tight estimates on the variance of the polymer weight constrained to stay within a thin on scale cylinder (i.e., an $n\times \theta n^{2/3}$ rectangle where $\theta$ is small but bounded away from 0), answering a question raised in \cite{DPM17} (see Proposition \ref{l:variancelb}). Studying fluctuations of polymers constrained to be in an on-scale rectangle \cite{BSS14} and showing that polymers constrained to be in a very thin rectangle is unlikely to be competitive with the unconstrained polymer \cite{BGH17} has been proved useful, however sharp results so far are mostly in the settings of below-scale thin rectangles \cite{CD13}, where $\theta=n^{-c}$ for some $c>0$.

\subsection{Key ideas and organization of the paper}\label{idea} 
Before jumping in to proofs, we present the key reasons driving the exponents and the main ingredients of the proofs of Theorem \ref{t:small} and \ref{t:large}. Since the reasons governing the behaviour of $\rho(n,r)$ when $r\ll n$ and $\rho(n,r)$ when $n-r\ll n$ are almost symmetric for brevity in this section we will mostly discuss the former case. It seems plausible that $\Gamma_n$ should overlap significantly with $\Gamma_r$  up to the region $\{x+y\le 2r\}$. Then at a very high level one can speculate that ${\rm Cov}(T_r,T_n)$ should be of the order of the variance of the amount of overlap which because of the previous sentence should be of the same order as ${\rm Var}(T_r)=O(r^{2/3})$ pointing towards a correlation of the order of $(\frac{r}{n})^{1/3}.$

We now mention a few key ingredients used to make the above heuristic rigorous. The upper bound is relatively straightforward. For convenience, as we shall do throughout the paper, let us denote $T_r$ by $X$ and let $T_n=Z+W$  where $Z$ is the length of the first part of the polymer $\Gamma_n$ i.e. the part from $\bf{0}$ to the line $x+y=2r$ and $W$ is the length of the path from $x+y=2r$ to $\bf{n}.$ See Figure \ref{fig.fluc0}. Let $v=(r+s,r-s)$ be the vertex at which $\Gamma_n$ intersects the line $x+y=2r$. It is well known since the work of Johansson \cite{J00} that if $r$ is say $n/2$ then $|s|=O(r^{2/3})$. However the polymer is in some sense self similar and hence one expects that the above result should also hold even when at scales $r \ll n$. Indeed a quantitative version of such a result was established in \cite{BSS17B}. This tells us that $|X-Z|=O(r^{1/3})$ by standard results about polymer fluctuations at scale $r^{2/3}$ around the point $\bf{r}$.  Moreover relying on this we also prove the local Brownian-like square root fluctuations  of the distance profile $T_{w,\bf{n}}$ as $w$ varies over vertices of the form $(r+s,r-s)$ when $|s|=O(r^{2/3})$ showing that $|W-Y|=O((r^{2/3})^{1/2})=O(r^{1/3})$ where $Y$ is $T_{{\bf r},{\bf n}}$ (hence is independent of $X$). Given the above information, the upper bound, i.e., Theorem \ref{t:small} is a simple consequence of Cauchy-Schwarz inequality. 

However the lower bound is significantly more delicate since one has to rule out cancellations to show that indeed the heuristic mentioned at the beginning of the section is correct. To do this the first thing to come to our aid is the FKG inequality. At a very high level the strategy is to condition on a large part of the noise space in a way which allows us to control cancellations and prove the desired lower bound on $\rho(n,r)$. Now if with positive probability $\beta$ (independent of $r,n$) the conditioned environment is such that $\rho(n,r)\ge \Theta(\frac{r}{n})^{1/3}$ then since $\rho(n,r)\ge 0$ pointwise on the  conditioned environment (using the FKG inequality), averaging over the latter yields the lower bound $\rho(n,r)\ge \beta \Theta(\frac{r}{n})^{1/3}$. Our strategy of choosing the part of the environment to condition on consists of ensuring that, with positive probability, the polymer $\Gamma_r$ is localized i.e., it is confined  to a thin cylinder $R_{\theta}$ of size $r\times \theta r^{2/3}$ for some small $\theta$ and ensuring $\Gamma_n$ essentially agrees with  $\Gamma_r$ up to the line $x+y=2r.$ This is obtained by creating a bad region (barrier) around the thin cylinder making it suboptimal for the polymer to venture out of $R_\theta$. This then implies that under such a conditioning, up to correction terms ${\rm Cov}(T_r,T_n)$ is $\Var(T_r)$. As mentioned before, at this point we prove a sharp estimate on variance of polymer weights constrained to lie in $R_{\theta}$ showing that it scales like $\theta^{-1/2} r^{2/3}$ as $\theta$ goes to $0$. Thus for $\theta$ small enough, the variance term is large enough and dominates all the correction terms yielding the sought lower bound of $\Theta(r^{2/3})$ on the covariance and hence Theorem \ref{t:large}.
\begin{figure*}
\begin{multicols}{2}
 \includegraphics[width=.40\textwidth]{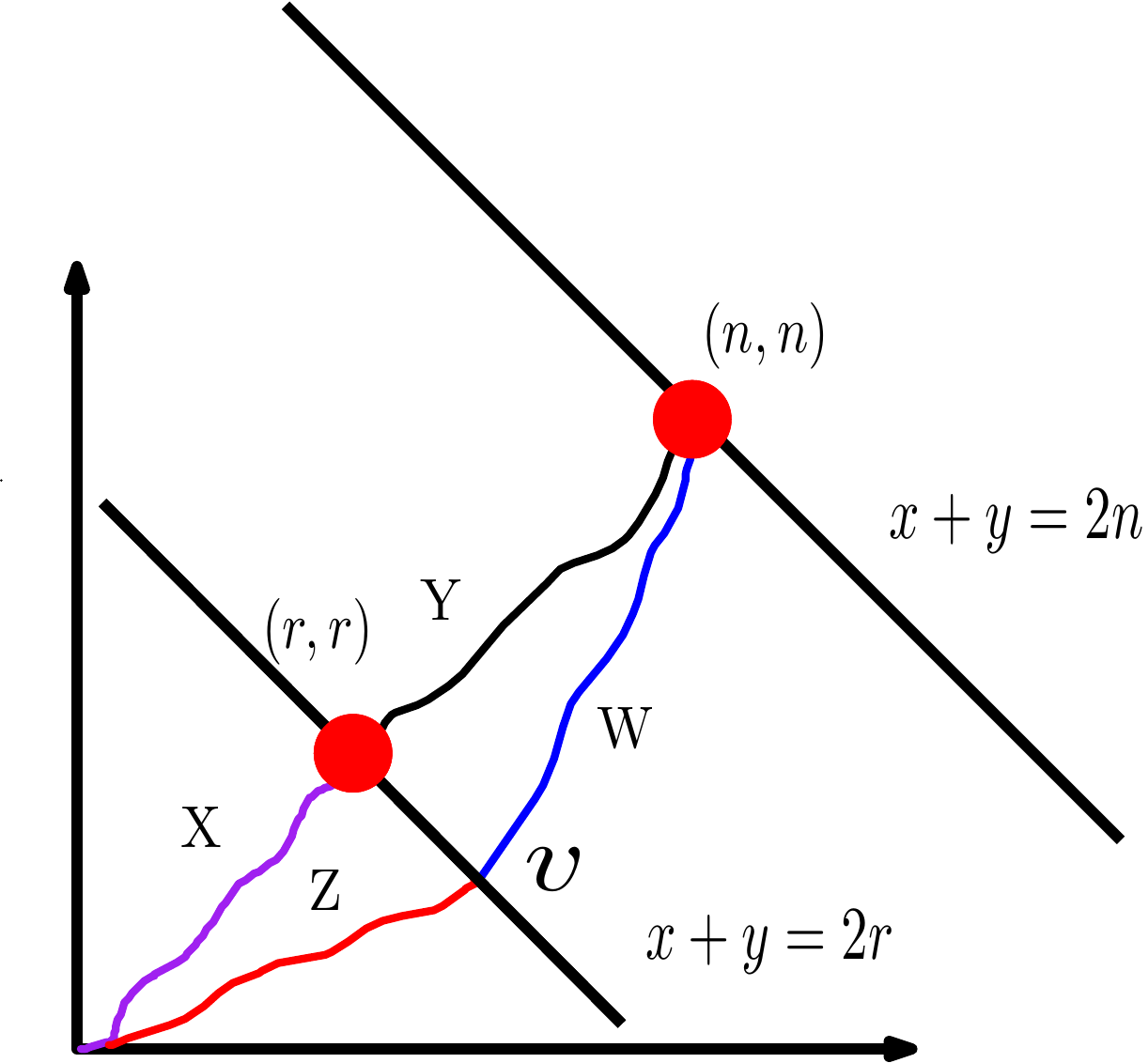}\par 
 \includegraphics[width=.40\textwidth]{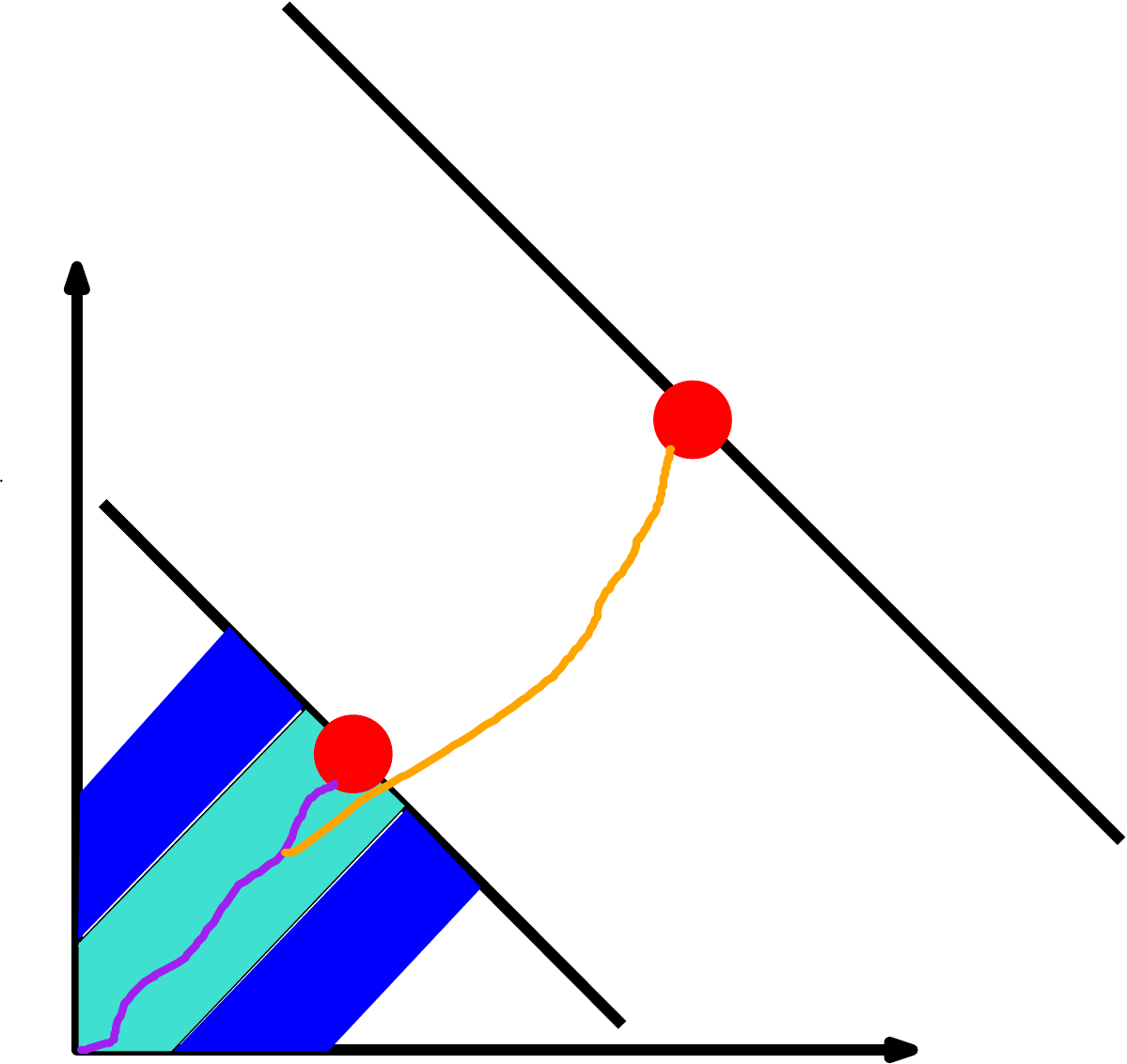}  
     \end{multicols}
%
\caption{The figure illustrates the polymers of interest, $\Gamma_{r}$ with weight $X$, $\Gamma_{\bf{r},\bf{n}},$ with weight $Y,$ $\Gamma_{n}$ comprised of $\Gamma_{\bf{0},v}$ and $\Gamma_{v,\bf{n}}$ with weights $Z$ and $W$ respectively where $v$ is the point of intersection of $\Gamma_n$ with the line $x+y=2r.$  The second figure illustrates our strategy to create barriers (deep blue) around a narrow strip (light blue) to ensure that $\Gamma_r$ and $\Gamma_{\bf{r},v}$ stay localized inside the latter and hence overlaps significantly creating a situation where the covariance between $X$ and $Z+W$ is approximated by the variance of the former.}
\label{fig.fluc0}
\end{figure*}

We now briefly describe how to use the exact same strategy to bound $\rho(r,n)$ in the regime $n-r\ll n.$ We will discuss the more delicate Theorem \ref{t:large}.  Note that in this case we are aiming to prove a lower bound on $1-\rho(n,r)$ and hence an upper bound on $\rho(n,r).$  Thus the natural strategy to adopt would be to show that even after conditioning on $T_r$, $T_n$ is not completely determined and there is still some fluctuation left. In fact, as expected, our arguments  will show that the latter is of the same order as the fluctuation of $T_{{\bf r}, {\bf n}}$ i.e., ${\rm{Var}}(T_n\vert T_r)=\Theta((n-r)^{2/3})$ {on a positive measure part of the space.} 
Thus we get 
$$\Theta((n-r)^{2/3})\le \inf_{\lambda}{\rm Var}(T_n-\lambda T_r)= 
 (1-\mbox{Corr}^2(T_r,T_n)){\rm {Var}}(T_n).$$
This, along with the fact that ${\rm {Var}}(T_n)=\Theta(n^{2/3})$, completes the proof. 

{Before moving further, we wish to point out that, while we do crucially make use of the integrability of the exponential LPP model, we only use the integrable input of weak convergence to Tracy-Widom distribution \cite{Jo99} and the moderate deviation estimates coming from \cite{BFP12} (see Theorem \ref{t:moddev}). Therefore we expect our methods to be applicable to a large class of integrable LPP models where such estimates are known. In particular, we do not use any information about the limiting Airy process. As already mentioned, our approach hinges on fine understanding of the local polymer geometry, following the sequence of recent works \cite{BSS14, BSS17B, BGH17}. Indeed extensively draw from some of the estimates derived in those previous works, while introducing some new elements to better the understanding of polymer geometry. By virtue of being geometric, our proof is also robust, i.e., we expect to be able use it to control correlation between polymer weights with endpoints near the polymers considered in this paper.}

\subsection*{Organization of the paper}
The rest of the paper is organized as follows. We first prove Theorem \ref{t:aest} in Section \ref{s:a}. Then we use Theorem \ref{t:aest} to prove Theorem \ref{t:small} in Section \ref{s:easy}. The strategy of proof for Theorem \ref{t:large} is described in Section \ref{s:dis}, where the proof is completed modulo the key Proposition \ref{p:positive}. Proposition \ref{p:positive} is proved in Section \ref{s:positive} using Lemma \ref{geodimp} and Lemma \ref{l:realvar} whose proofs are postponed to Section \ref{s:l2} and Section \ref{s:var} respectively. Throughout the paper we use a number of auxiliary results about fluctuations of polymer lengths between pairs of points in an on-scale rectangle, which were derived in \cite{BSS14}. For easy reference purpose we collect these results in Appendix \ref{s:appa}. Some consequences of these estimates which are also useful for us are recorded in Appendix \ref{s:appb}.

\subsection*{Acknowledgements}
We thank Alan Hammond and Ivan Corwin for discussing the results in \cite{CH14+}, and Alan Hammond for extensive discussions around the results of \cite{H16} and Theorem \ref{t:aest}. RB's research is partially supported by a Ramanujan Fellowship from Govt.\ of India and an ICTS-Simons Junior Faculty Fellowship from Simons Foundation. SG's research is partially supported by a Miller Research Fellowship.  

\section{Local fluctuations of weight profile}
\label{s:a}
We prove Theorem \ref{t:aest} in this section. As explained before, if $s=\Theta(n^{2/3})$ one can read off a qualitative version of this result from the limiting Airy process. However if $s\ll n^{2/3}$ the local fluctuation information disappears in the Airy process limit. Although that Brownian motion arises as a week limit at some shorter scale is known \cite{P17}, we need some finer estimates for finite $n$. This was achieved in \cite{H16} using the Brownian Gibbs property of the pre-limiting line ensemble in Brownian LPP. We shall take a more geometric approach which hinges on establishing that the profile $\{L_{n,s'}-L_{n,0}: |s'|<s\}$ is with high probability determined by the vertex weights in the region 
$$\{(x,y): 2n-C_{*}s^{3/2} \leq (x+y)\leq 2n\}$$
for some large constant $C_*$. To this end we have the following proposition. 

\begin{proposition}
\label{p:car}
In the set-up of Theorem \ref{t:aest}, consider $\Gamma=\Gamma_{\mathbf{0}, (n+s',n-s')}$. Let $t\geq 1$ and let $v=(v_1(s',t,s),v_2(s',t,s))$ denote the point at with $\Gamma$ intersects the anti-diagonal $x+y=2n-2t^*s^{3/2}$. There exists $s_0>0,y_0>0$ and $c,C>0$ such that the following holds for all $s>s_0$, $t\geq 1, y>y_0$ and for all $n> Cr^{3/2}$:
$$\P\left (\sup_{|s'|<s} |v_1(s',t,s)-(n-t^{*}s^{3/2})|\geq y t^{2/3}s \right) \leq e^{-cy^2}$$
where $t^{*}=\min\{t, \frac{n}{s^{3/2}}\}$. 
\end{proposition}

See Figure \ref{fig.fluc1} for an illustration of the large probability event in Proposition \ref{p:car}.

\begin{figure}[htbp!]
\includegraphics[width=.60\textwidth]{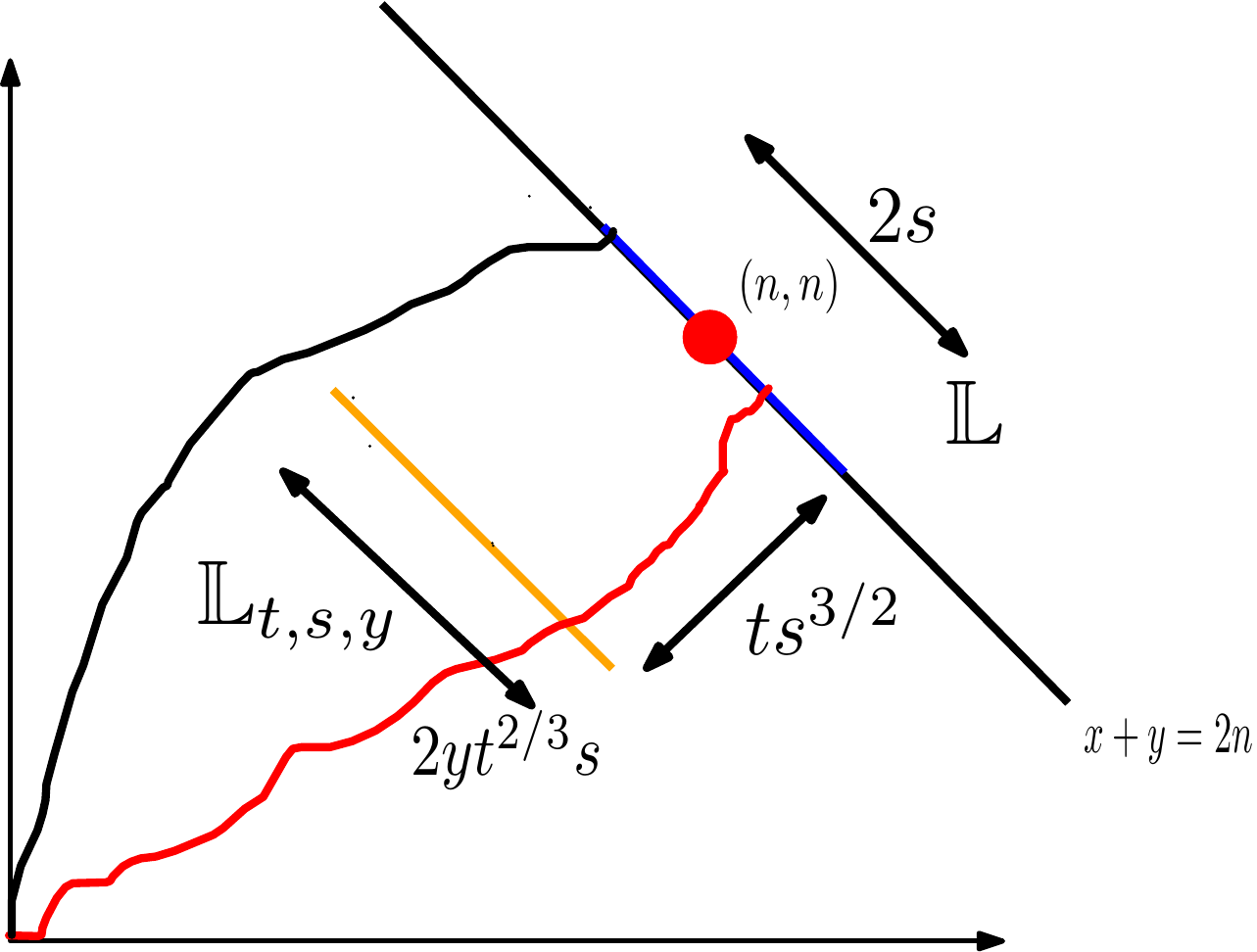} 
\caption{The figure illustrates the setting of Proposition \ref{p:car} where the geodesic from $\bf{0}$ to any point in an interval $\mathbb{L}$ of length $2s$ centered at $\mathbf{n}$ on the line $x+y=2n$ is unlikely to intersect the line $x+y=2(n-ts^{3/2})$ outside the line segment $\mathbb{L}_{t,s,y}$ of length $2yt^{2/3}s$ centered at $(n-ts^{3/2},n-ts^{3/2}).$ Thus it is unlikely that the black path would be a geodesic.  }
\label{fig.fluc1}
\end{figure}

\begin{proof}
Clearly, the case $t^*\neq t$ is trivial. For the other case, observe first that by polymer ordering, it suffices to prove the result for $s'=\pm r$. This case can be read off from the proof of Theorem 2 in \cite{BSS17B}(see also Remark 1.5 there).
\end{proof}

As in Figure \ref{fig.fluc1}, let $\mathbb{L}$ denote the line segment joining $(n+s,n-s)$ and $(n-s,n+s)$ and $\mathbb{L}_{t,s,y}$ denote the line segment joining $(n-t^*s^{3/2}-yt^{2/3}s, n-t^*s^{3/2}+yt^{2/3}s)$ and $(n-t^*s^{3/2}+yt^{2/3}s, n-t^*s^{3/2}-yt^{2/3}s)$ . Clearly on the large probability event (for large $y$) given by Proposition \ref{p:car}, the profile: $$\{L_{n,s'}-L_{n,0}: |s'|<s\}$$ can be upper bounded by using the passage times $T_{u,v}$ where $u\in \mathbb{L}_{t,s,y}$ and $v\in \mathbb{L}$. The next proposition shows that these passage times are concentrated around their expectations with sufficiently high probability.

\begin{proposition}
\label{p:sidetoside}
Let $\delta\in (0,\frac{1}{3})$ be fixed. Set $\mathbb{L}':=\mathbb{L}_{t,s,t^{\delta}}$. Then there exists $s_0,y_0>0$ and $c>0$ such that for all $s>s_0, y>y_0$ and $t\geq 1$ we have 
$$ \P\left( {\sup_{u\in \mathbb{L}', v\in \mathbb{L}} |T_{u,v}-\E T_{u,v}| \geq yt^{1/3+\delta/2}s^{1/2}} \right) \leq e^{-cy}.$$
The same bound holds for $u=\mathbf{0}$ if $t\neq t^*$.
\end{proposition}

\begin{proof}
This follows from Proposition \ref{t:treeinf} and Proposition \ref{t:treesup} by observing that the slope between any two pair of points in $\mathbb{L}$ and $\mathbb{L}'$ remain between $1/2$ and $2$ ({here we use the fact that $\delta< \frac{1}{3}$, and $s$ is sufficiently large}).
\end{proof}

The next lemma gives a comparison of $\E T_{u,v}$ and $\E T_{u,v'}$ as $u$ is kept fixed on $\mathbb{L}'$ and $v$ varies on $\mathbb{L}$. The proof is an easy computation using Theorem \ref{t:moddev} and is omitted. 

\begin{lemma}
\label{l:mean}
In the set-up of Proposition \ref{p:sidetoside}, there exists $C_0>0$ such that for all $t>1$ we have 
$\sup_{u\in \mathbb{L}', v,v'\in \mathbb{L}} \E X_{u,v}-\E X_{u,v'}\leq C_0s^{1/2}$. The same conclusion remains for $u=\mathbf{0}$, if $t\neq t^*$.
\end{lemma}

We can now complete the proof of Theorem \ref{t:aest}. 

\begin{proof}[Proof of Theorem \ref{t:aest}]
Let $z>0$ sufficiently large be fixed. Let $t=z^{4/3}$ and let $\mathcal{A}$ denote the event that 
$$\left \{\sup_{s': |s'|<s} |v_1(s',t, s)-(n-t^*s^{3/2})|\geq z^{10/9}s\right\}$$
Use Proposition \ref{p:car} with the above value of $t$ and $y=z^{2/9}$ to conclude that $\P(\mathcal{A})\leq e^{-cz^{4/9}}$ (as $z$ is sufficiently large).
We shall now consider two cases separately: (i) $t=t^*$ and (ii) $t\neq t^*$.

In case (i), let $\mathbb{L}'$ be defined as in Proposition \ref{p:sidetoside} with $\delta=1/6$ and the choice of $t$ as before. Let $\mathcal{B}$ denote the event that 
$$\left\{ {\sup_{u\in \mathbb{L}', v, v'\in \mathbb{L}} |T_{u,v}-T_{u,v'}| \geq zs^{1/2}} \right\}.$$
Using Proposition \ref{p:sidetoside} with the choices above, $y=z^{4/9}$ and Lemma \ref{l:mean} we get that for $z$ sufficiently large, we have $\P(\mathcal{B})\leq e^{-cz^{4/9}}$. It remains to prove that on $\mathcal{A}^{c} \cap \mathcal{B}^c$ we have 
$$\left\{\sup_{s': |s'|<s} L_{n,s'}-L_{n,0} \leq  zs^{1/2}\right\}.$$

To see this, let $s_*$ with $|s_*|\leq s$ be such that 
$$L_{n, s_*}=\sup_{s': |s'|<s} L_{n,s'}.$$
Let $v:= (n+s_*, n-s_*)$. On $\mathcal{A}^c$, the geodesic $\Gamma_{\mathbf{0}, (n+s_*,n-s_*)}$ intersects the line segment $\mathbb{L}'$, let $u_*$ be the intersection point. On $\mathcal{B}^c$ we have $|T_{u_*,v}-T_{u_*, \mathbf{n}}|\leq zs^{1/2}$. The claim is established by observing that $L_{n,0}\geq T_{\mathbf{0}, u_*}+ T_{u_*, \mathbf{n}}$.

In case (ii), we proceed as before but now notice that $u_*=\mathbf{0}$. The same argument now can be repeated with $\mathcal{B}'$ defined as 
$$\left\{ {\sup_{v, v'\in \mathbb{L}} |T_{\mathbf{0},v}-T_{\mathbf{0},v'}| \geq zs^{1/2}} \right\};$$
and using the second parts of Lemma \ref{l:mean} and Proposition \ref{p:sidetoside} to show that $\P(\mathcal{B}')\leq e^{-cz^{4/9}}$. The rest of the proof is identical with the previous case.
\end{proof}

\section{Proof of upper bounds}
\label{s:easy}

In this section we shall prove Theorem \ref{t:small} using Theorem \ref{t:aest} and Proposition \ref{p:car}. As we shall see, the proofs of parts (i) and (ii) will depend on much of the same ingredients. Before proceeding further let us introduce some notation that will be used throughout this section.

Before diving in to the proofs we adopt the convention of ignoring the values of the vertices $\{\omega_{(x,y)}: x+y=2r\}$. This would enable us to write cleaner equations of the form $T_{\bf{n}}=T_{{\bf 0}, v}+ T_{v,{\bf n}}$ where $v$ is the unique vertex $\Gamma_n \cap \{x+y=2r\}.$
However since by definition, the random $v$ can be one of $2r$ possible vertices, whose maximum value is no more than $\log r$ with exponential tail, it does not create any change in the computations throughout the paper since all the objects that we deal with, have  fluctuations of the order of $r^{1/3}$. We shall adopt this convention throughout the remainder of this paper, and not comment further on this topic. It will be easy to verify the minor details in each case, and we leave that to the reader.

For any path $\gamma$, we shall denote by $\ell(\gamma)$, the weight of the path. Let $\Gamma:= \Gamma_{n}$ denote the polymer from $\mathbf{0}$ to $\mathbf{n}$. Let $v=(v_1,v_2)$ denote the point at which $\Gamma$ intersects the line $\{x+y=2r\}$. Recall from Section \ref{first} that $T_{r}:=T_{\mathbf{0}, \mathbf{r}}$.
Let us define (see Figure \ref{fig.fluc0}) 
\begin{align}\label{labels}
X&:=T_{r},\,\,\,
Y:=T_{\mathbf{r},\mathbf{n}}, \\ 
\nonumber
Z&:=\ell(\Gamma_{\mathbf{0},v}) \text{ and }  W:=\ell(\Gamma_{v,\mathbf{n}}). 
\end{align}
Thus by definition $T_{n}=Z+W$ \footnote{This is first of the many situations we ignore the weights on the line $x+y=2r$, as mentioned above we shall not comment on this issue henceforth.}. Finally we shall denote by $X^*$ the weight of the polymer, denoted by $\Gamma^*$, from $\mathbf{0}$ to the line $\{x+y=2r\}$. 

We shall need some preparatory results. First we want to show that $(Z-X)_{+}$ is tight at scale $r^{1/3}$. Observing that $X^*\geq Z$ this is a consequence of Lemma \ref{l:p2l}.
The next lemma shall show that $W-Y$ is also typically of order $r^{1/3}$. Notice that if $r\ll n$, now we can no-longer replace $W$ by the weight of the line-to-point polymer from the line $\{x+y=2r\}$ to $\mathbf{n}$. This is where we shall need the full power of Proposition \ref{p:car} and Theorem \ref{t:aest}. 

\begin{lemma}
\label{l:yw}
There exists positive constants $r_0,y_0$ and $C,c>0$ such that for all $r>r_0$ and $y>y_0$ and $n>Cr$ we have
$$\P(W-Y>yr^{1/3})\leq e^{-cy^{1/3}}.$$
\end{lemma}

\begin{proof}
For $z>0$, let $\mathcal{A}_{z}$ denote the event $|v_1-r|\geq zr^{2/3}$ and $\mathcal{B}_{z}$ denote the event that 
$$ \sup_{|s|\leq zr^{2/3}} T_{(r+s,r-s),\mathbf{n}} -T_{\mathbf{r},\mathbf{n}} \geq yr^{1/3}.$$
Clearly for every $z>0$,
$$\P(W-Y>yr^{1/3})\leq \P(\mathcal{A}_z)+\P(\mathcal{B}_{z}).$$
The lemma follows by taking $z=y^{1/6}$ and using Proposition \ref{p:car} and Theorem \ref{t:aest} to bound $\P(\mathcal{A}_z)$ and $\P(\mathcal{B}_{z})$ respectively. {Note that in the last application, Theorem \ref{t:aest} is applied for the inverted ensemble i.e., replace $\bf{n}$ by $\bf 0$ and $\bf r$ by $\bf n-r$} .
\end{proof}

We can now prove the following proposition which immediately implies Theorem \ref{t:small}, (i) as ${\rm {Var}} ~T_{n}= \Theta (n^{2/3})$, and ${\rm {Var}} ~T_{r}= \Theta (r^{2/3})$.

\begin{proposition}
\label{p:easyu}
There exists absolute constants $r_0, \delta_1$ and $C$ such that we have for all $r_0<r<\delta_1 n$ and $n$ sufficiently large 
$${\rm {Cov}} (T_{n}, T_{r}) \leq Cr^{2/3}.$$
\end{proposition}

\begin{proof}
We need to upper bound 
$${\rm {Cov}} (X, Z+W)= {\rm {Cov}}(X,Z)+ {\rm {Cov}}(X,W).$$
we bound the two terms separately. Clearly, by Cauchy-Schwarz inequality and the observation ${\rm {Var}}~ X= \Theta (r^{2/3})$, to prove ${\rm {Cov}}(X,Z)\leq Cr^{2/3}$, it suffices to show that ${\rm {Var}}~Z=O(r^{2/3})$. 
Now notice that, ${\rm {Var}}(Z) \leq 2({\rm {Var}} X + \E (X-Z)^2)$. Observing $ Y-W\leq Z-X \leq X^*-X $, and using Lemma \ref{l:p2l} and Lemma \ref{l:yw} it follows that $\E (X-Z)^2= O(r^{2/3})$ which in turn implies ${\rm {Cov}}(X,Z)\leq Cr^{2/3}$ for some absolute constant $C$.

For the second term in the above decomposition observe that 
$${\rm {Cov}}(X,W)={\rm {Cov}}(X,W-Y)$$
because $X$ and $Y$ are independent. Using Cauchy-Schwarz inequality again, it suffices to show that $\E (Y-W)^2=O(r^{2/3})$. Observing as before that $W-Y \geq X-Z \geq X-X^*$, this follows from Lemma \ref{l:p2l} and Lemma \ref{l:yw}. This completes the proof of the proposition.
\end{proof}

To prove Theorem \ref{t:small}, (ii) we shall need the following easy observation. 

\begin{observation}
\label{o:corr}
For any two random variables $U$ and $V$ we have 
$${\mathrm {Var}}(U-V)\geq (1-{\rm{Corr}}^2(U,V)){\mathrm {Var}}(U).$$
\end{observation}

The observation follows from noticing that 
\begin{align*}
 (1-\mbox{Corr}^2(U,V)){\rm {Var}}(U)=\inf_{\lambda\in \R} \Var(U-\lambda V)\le {\rm {Var}}(U-V).
\end{align*}

Using Observation \ref{o:corr}, the following Proposition immediately implies Theorem \ref{t:small}, (ii).

\begin{proposition}
\label{p:easyl}
There exists $r_0\in \N$ and positive absolute constants $\delta_1$, $C$, for all $r$ such that $\delta_1 n>(n-r)>r_0$, and all $n$ sufficiently large we have 
$${\rm {Var}}~(T_{n}-T_{r})\leq C(n-r)^{2/3}.$$
\end{proposition}

\begin{proof}
Recall $X,Y,Z,W$ as defined at the beginning of this section. We need to upper bound ${\rm {Var}} (Z+W-X)$. Expanding we get that, 
$${\rm {Var}} (Z+W-X)= {\rm {Var}}(Z-X) +{\rm {Var}}(W)+ 2{\rm {Cov}} (Z-X,W).$$
We shall show each of the terms above is $O((n-r)^{2/3})$ separately. In fact, by Cauchy-Schwarz inequality it suffices to only show that bound for the first two terms. Notice that the picture is same as before except the roles of $r$ and $(n-r)$ has been reversed. Using the proof of Lemma \ref{l:yw} we can now show that 
$$ \P(Z-X \geq y (n-r)^{1/3})\leq e^{-cy^{1/4}}$$
and using the proof of Lemma \ref{l:p2l} it follows that 
$$\P(W-Y \geq y (n-r)^{1/3})\leq e^{-cy^{1/4}}$$
for all $y$ sufficiently large. 
As in the proof of Proposition \ref{p:easyu}, this is then used to argue that 
${\rm {Var}}(Z-X)=O((n-r)^{2/3})$, and $\E[W-Y]^2= O((n-r)^{2/3})$, which together with the observation that $\Var~{Y}=O((n-r)^{2/3})$ completes the proof of the proposition. 
\end{proof}

\section{Proof of lower bounds} 
\label{s:dis}
We now move towards proving Theorem \ref{t:large}. As in the proof of Theorem \ref{t:small}, parts (i) and (ii) of Theorem \ref{t:large} have rather similar proofs as well (after exchanging the roles of $r$ and $n-r$). In this section we describe in detail the line of argument leading to the proof of Theorem \ref{t:large}, (i) in detail, and complete the proof modulo the key result Proposition \ref{p:positive}. At the end of the section we shall give a sketch of how the same strategy is used to prove Theorem \ref{t:large}, (ii).

For the readers' benefit we recall briefly the strategy outlined in Section \ref{idea}. By the FKG inequality, it should suffice to obtain a lower bound on the conditional correlation on an event with probability bounded uniformly below. By the trivial observation ${\rm {Cov}}(X,X+Y)=\Theta(r^{2/3})$, a very natural way to construct such an event is to ask that $v$ is very close to $\mathbf{r}$ which will imply $X\approx Y$ and $Z\approx W$ (using Theorem \ref{t:aest}). However one needs to be careful so that there will be enough fluctuation left in the conditional environment. To this end, it turns out, one can construct such an event  measurable with respect to the configuration outside a thin strip of width $\Theta(r^{2/3})$ around the straightline joining $\mathbf{0}$ to $\mathbf{r}$.

For $\theta>0$, let $R_{\theta}\subseteq \Z^2$ be defined as follows:
$$R_{\theta}:=\{(x,y)\in \Z^2: 0\leq x+y\leq 2r~\text{and}~|x-y|\leq \theta r^{2/3}\}.$$
Let $\omega_{\theta}=\{\omega_{v}:v\in \llbracket 0,n \rrbracket ^2 \setminus R_{\theta}\}$
denote a weight configuration outside $R_{\theta}$. Let $\cf_{\theta}$ denote the $\sigma$-algebra generated by the set of all such configurations $\Omega_{\theta}$. Observe that events measurable with respect to $\cf_{\theta}$ can be written as subsets of $\Omega_{\theta}$, and we shall often adopt this interpretation without explicitly mentioning it.
The major step in the proof is the following proposition. 

\begin{proposition}
\label{p:event}
There exist absolute positive constants $\beta, \delta_1, \theta, C>0$ sufficiently small such that for $\delta_1 n>r>r_0$ and $n$ sufficiently large there exists an event $\ce$ measurable with respect to $\cf_{\theta}$ with $\P(\ce)\geq \beta$ and the following property: for all weight configuration $\omega\in \ce$ we have 
$${\rm {Cov}} (T_{n}, T_{r}\mid \omega) > Cr^{2/3}.$$
\end{proposition}

The proof of Theorem \ref{t:large}, (i) using Proposition \ref{p:event} is straightforward. 

\begin{proof}[Proof of Theorem \ref{t:large}, (i)]
Observe that for each fixed weight configuration $\omega\subset \Omega_{\theta}$ on the vertices outside $R_{\theta}$, both $T_{n}$ and $T_{r}$ are increasing in the weight configuration on $R_{\theta}$. Observe also that $\E[T_{n}\mid \cf_{\theta}]$ and $\E[T_{r}\mid \cf_{\theta}]$ are both again increasing in the configuration $\omega$. Applying the FKG inequality twice together with Proposition \ref{p:event} then implies

\begin{eqnarray*}
\E T_{n} T_{r} &=& \E \left( \E [T_{n} T_{r} \mid \cf_{\theta}]\right)\\
&=& \int_{\ce}  \E [T_{n} T_{r} \mid \cf_{\theta}] d\omega + \int_{\ce ^c}  \E [T_{n} T_{r} \mid \cf_{\theta}] d\omega\\
&\geq &  \int_{\ce}  \E [T_{n}\mid \cf_{\theta}] \E[T_{r} \mid \cf_{\theta}] d\omega + \int_{\ce ^c}  \E [T_{n}\mid \cf_{\theta}] \E[T_{r} \mid \cf_{\theta}] d\omega + C\beta r^{2/3}\\
&{=} & \E\left( \E [T_{n}\mid \cf_{\theta}] \E[T_{r} \mid \cf_{\theta}]\right)+C\beta r^{2/3}\\
&\geq & \E[T_n]\E[T_{r}]+C\beta r^{2/3};
\end{eqnarray*}
which is what we set out to prove.
\end{proof}

\subsection{Constructing a suitable environment}
The key of the proof is construction of $\ce$, towards which we now move. For easy reference we recall the notations already introduced in Section \ref{s:easy}, that we will use again.
\begin{align*}
X&:=T_{r},\,\,\,
Y:=T_{\mathbf{r},\mathbf{n}},  
Z:=\ell(\Gamma_{\mathbf{0},v}) \text{ and }  W:=\ell(\Gamma_{v,\mathbf{n}}).\\
X^*&:=\max \{\ell(\Gamma_{\mathbf{0},v}): v\in \mathbb{L}_{r}\},
\end{align*}
where $\mathbb{L}_{r}$ denote the line $\{x+y=2r\}$. We shall also denote by $X_*$ (resp.\ $X_{\theta}$) the length of the best path from $\mathbf{0}$ to $\mathbf{r}$ that does not exit $R_{2\theta}$ (resp.\ $R_{\theta}$). Finally for $\phi >\theta$, $\mathbb{L}_{r,\phi}$ shall denote the line segment joining $(r-\phi r^{2/3}, r+\phi r^{2/3})$ and $(r+\phi r^{2/3}, r-\phi r^{2/3})$. We shall denote by $X_{\phi}$ the length of the best path from $\mathbf{0}$ to $\mathbb{L}_{r,\phi}$, and by $Y_{\phi}$ the length of the best path from $\mathbb{L}_{r,\phi}$ to $\mathbf{n}$.

The event $\ce$ will depend on a number of parameters $\phi_0, \phi_1, \phi, L, c_0, C_0, C^*$ (and naturally $\theta$), the choices of which shall be specified later. The event will consist of two major parts.

\begin{enumerate}
\item \textbf{Regular fluctuation of the profile $\{T_{v,\mathbf{n}}:w\in \mathbb{L}_{r,\phi}\}$:} Let $\ce_1$ denote the event that 
\begin{align*}
\left\{ \sup_{w\in \mathbb{L}_{r,\phi_0}} T_{w,\mathbf{n}}- T_{\mathbf{r}, \mathbf{n}} \leq \phi_0^{1/2}\log ^9 (\theta^{-1}) r^{1/3}\right\} \cap \left\{ \sup_{v\in \mathbb{L}_{r,\phi}\setminus  \mathbb{L}_{r,\phi_0}} T_{w,\mathbf{n}}- \sqrt{|w_1-w_2|}\log^{9}(\theta^{-1}) \leq T_{\mathbf{r}, \mathbf{n}} \right\},
\end{align*}
where $w=(w_1,w_2)$. Observe that $\ce_1$ only depends on the weight configuration above the line $\mathbb{L}_{r}$.

\item \textbf{Barrier around $R_{\theta}$:} Let $U_1$ (resp.\ $U_2$) denote a $r\times (\phi-\theta)r^{2/3}$ rectangle whose one set of parallel sides are aligned with the lines $x+y=0$ and $x+y=2r$ respectively and whose left (resp.\ left right) side coincides with the right (resp.\ left) side of $R_{\theta}$\footnote{Here is another convention we shall adopt throughout. In keeping with the often used practice or rotating the picture counter-clockwise by 45 degrees, so that the line $x=y$ becomes vertical, we shall often refer to sides of rectangles parallel to it by `left' and `right', and sides perpendicular to it by `top' and `bottom'.}. For any point $u=(u_1,u_2)\in \Z^2$, let $d(u):=u_1+u_2$. Also, for any region $U$, and points $u,v\in U$, let us denote, by $T_{u,v}^{U}$ to be the length of the longest path from $u$ to $v$ that does not exit $U$. Let $\ce_2$ denote the following event measurable with respect to the configuration in $U_1$: 
$$ T^{U_1}_{u,u'}-\E T_{u,u'} \leq -Lr^{1/3} ~\forall u,u' \in  U_1 ~\text{with}~ |d(u)-d(u')|\geq \frac{r}{L}.$$
Let $\ce_{3}$ denote the same event with $U_1$ replaced by $U_2$. We set $\ce_4:=\ce_2\cap \ce_3$.
\end{enumerate}

Before defining $\ce$ we shall need to define a few events  that are contained in $\ce_4$.
We start with some notational preparation.  As before $\L_r=\{(x,y): x+y=2r\}.$
For any $w \in \L_r$ let $$\cB_w \text{ denote the event that } \ell (\Gamma_{{\bf{0}},w})\ge X_{\theta}.$$
Moreover let $v^*$ be the point at which $\Gamma^*$ intersects $\L_{r}$. Also recall from \eqref{labels} that $v$ denotes the vertex at which $\Gamma_n$ intersects $\L_r$. As before, for any $\kappa>0$ we will denote by $\L_{r,\kappa}$ the line segment $\{(x,y): x+y=2r\}\cap \{(x,y): |x-y|\leq \kappa r^{2/3}\}.$
Moreover let $$\cB^{\kappa}=\bigcup_{w \in \L_r \setminus \L_{r,\kappa}}\cB_w,$$ i.e., there exists some polymer starting at $0$ and ending at $w \in \L_r \setminus \L_{r,\kappa}$ with length as large as $X_{\theta}$. Similarly let $\cB^{\kappa}_*$ denote the event that $v^* \in \L_r \setminus \L_{r,\kappa}.$
Note that  
$\cB^{\kappa}_* \subseteq \cB^{\kappa}.$ Along the same lines let $\cA^{\kappa}_{\rm loc}$ denote the event that 
$$\sup_{w\in \L_{r}\setminus \L_{r,\kappa}} T_{{\bf 0}, w}+T_{w, {\bf{n}}}\ge Y+X_{\theta}.$$

\subsubsection{Choice of parameters}
\label{s:parameters}
We need to fix our choice of parameters appearing in the definitions of the above events before proceeding to proving probability bounds for the same. 
{Throughout the sequel $c_0$ is a small enough universal constant, $C_1$ and $C^*$ will be sufficiently large universal constants. We shall choose $\theta$ to be an arbitrarily small constant; and $L\gg \phi \gg \phi_1 \gg \phi_0$. We need to choose $\phi_0$ poly-logarithmic in $\theta^{-1}$, $\phi_1$ a large inverse power of $\theta$, $\phi$ a large inverse power of $\theta$ depending on $\phi_1$, and $L$ a much larger inverse power of $\theta$ depending on $\phi$. For concreteness we shall fix $\phi_0=\log^{10}(\frac{1}{\theta}),$ $\phi_1=(\frac{1}{\theta})^{10},$ $\phi=(\frac{1}{\theta})^{100}$ and $L=\frac{1}{\theta^{500}}$. Given all of these we shall take $r$ sufficiently large, and $r/n$ sufficiently small}. Throughout the remainder of this paper we shall work with this fixed choice of parameters. 

\subsection{Some consequences of conditioning on  $\ce_4$}
We shall now prove probability bounds of certain events (which will later be used to analyze $\ce$) conditional on $\ce_{4}$ for the above choices of parameters.

\begin{lemma}\label{macrolocal} For all $\theta$ small enough and $\phi_1$ defined as above:
\begin{align}
\label{cond1}
\P(\cB^{\phi_1}\mid \ce_4)&\le e^{-\frac{1}{\theta^2}},\\
\label{cond2}
\P(\cA^{\phi_1}_{\rm loc}\mid \ce_4)&\le e^{-\frac{1}{\theta^2}}.
\end{align}
\end{lemma}
\begin{proof} Both the events are negative in the environment outside $R_\theta$. Thus it suffices to upper bound $\P(\cB^{\phi_1})$  and $\P(\cA^{\phi_1}_{\rm loc})$ by the FKG inequality. Observe that to prove \eqref{cond1}, we need to show that it is unlikely that $X_{\theta}$ is too small and it also that $\sup_{w\in \L_{r}\setminus \L_{r,\phi_1}}$ is too large. 
The first part follows from the proof of Lemma \ref{l:4th2} (which shows that typically $X_{\theta}-4r=-\Theta(\theta^{-1}r^{1/3})$) and the second part follows as in the proof of Lemma \ref{l:p2l} by noticing that $\E T_{0,w}\leq 4r-c\phi_1^2 r^{1/3}$ for some $c>0$ and for each $w\in \L_{r}\setminus \L_{\phi_1,r}$, and our choice of $\phi_1$.

For the proof of \eqref{cond2}, observe that for any $w\in \L_{r}\setminus L_{r,\phi_1}$ using Theorem \ref{t:aest}, $\E T_{w,\mathbf{n}}-Y=O(\sqrt{|w_1-w_2|})$, whereas $\E T_{0,w}-X_{\theta} = O({|w_1-w_2|^{2}}/r)+O(\theta^{-1}r^{1/3})$ by Lemma \ref{l:p2l} and Lemma \ref{l:4th2}. By our choice of parameters, it follows that in expectation $T_{{\bf 0}, w}+T_{w, {\bf{n}}}-Y-X_{\theta}$ is a large negative quantity.

For the formal proof, set $S_j:=\L_{r,j+1}\setminus \L_{r,j}$. Let $\cC_{j}$ denote the event that $\sup_{w\in S_j} T_{{\bf 0}, w} -X_{\theta} \geq \inf_{w\in S_j} Y-T_{w, {\bf{n}}}$. Clearly, for $j>\phi_1$ we have can upper bound $\P(\cC_{j})$ by 
$$ \P(\sup_{w\in S_j} T_{{\bf 0}, w}-4r \geq -0.001 j^{2}r^{1/3}) +\P(X_{\theta}\leq 4r-0.001j^{2}r^{1/3})+\P(\inf_{w\in S_j} Y-T_{w, {\bf{n}}}\geq -0.002j^{2}r^{1/3}).$$

Using Theorem \ref{t:treesup} for $j<0.9 r^{1/3}$ and the crude estimate as in the proof of Lemma \ref{l:p2l} for $j\geq 0.9r^{1/3}$, we can show that the first probability is upper bounded by $e^{-cj^{2}}$. Arguing as in the proof of Lemma \ref{l:4th2} we get the second probability is upper bounded by $e^{-cj^2 \theta}$, whereas the third probability, by Theorem \ref{t:aest} is upper bounded by $e^{-cj^{2/3}}$. Summing over all $j>\phi_1$ and using that $\phi_1$ is a large power of $\theta^{-1}$ gives the result.  
\end{proof}

Now we say that conditional on $\ce_4$, with high probability the polymers from $0$ to every point in $L_{r,\phi_1}$ is contained in $L_{r,\phi}$. Towards this we first formally define transversal fluctuation of a path $\gamma$. For any path $\gamma$ from $0$ to $w$ for $w\in \L_{r,\phi_1}$ we define ${\rm TF}(\gamma):=\sup\{|x-y|: (x,y)\in \gamma\}$. Let $\cC:=\cC_{\phi_1}$ denote the event $$\sup\left \{\ell(\gamma)-X_{\theta}: {\rm{TF}}(\gamma)\ge  \frac{\phi}{2}r^{2/3}
 \right\}\ge -\theta^{-100}r^{1/3}$$  where the supremum is taken over all $\gamma$ joining $\bf 0$ to some point $w\in \L_{r,\phi_1}$. We have the following lemma. 
 
\begin{lemma}
\label{TFlocal} 
$\P(\cC\mid \ce_4)\le e^{-\frac{1}{\theta^2}}.$
\end{lemma}
\begin{proof} The event is again negative in the environment outside $R_{\theta}$. Thus by the FKG inequality it suffices to upper bound $\P(\cC)$. This follows by the upper bounding supremum of $\ell(\gamma)$ (where $\gamma$ varies over all paths having huge transversal fluctuation) using Proposition \ref{p:tfnew} and the lower bound on $X_{\theta}$ using Proposition \ref{t:treebox}, and by checking that the proof of Proposition \ref{p:tfnew} works with our choice of parameters, as $\phi$ is a sufficiently large power of $\phi_1$.
\end{proof}

The next lemma shows that on $\cE_4$, the point to line polymer weight $X^*$ is in fact 
well approximated by $X_* $ and moreover the end point $v^*$ which by \eqref{cond1} lies inside $\L_{r,\phi_1}$ with high probability in fact lies in $\L_{r,\phi_0}\subset \L_{r,\phi_1}$ with high probability. 
To proceed we define an event similar to $\cC$ but considering paths that have transversal fluctuation less than $\frac{\phi}{2} r^{2/3}.$

More precisely let $\cD:=\cD_{\phi_1}$ be the event:
$$\sup_{w,\gamma} \left \{\ell(\gamma)-X_{*}+\frac{|w_1-w_2|^{2}}{100 r \theta^{3/2} }:  {\rm{TF}}(\gamma)\le  \frac{\phi}{2}r^{2/3}
 \right\}\ge 0$$   where the supremum is taken over $w=(w_1,w_2)\in \L_{r,\phi_1}\setminus \L_{r,\phi_0}$ and all $\gamma$ joining $\mathbf{0}$ and $w$. We have the following lemma.

\begin{lemma}
\label{geodimp} 
For all $\theta$ sufficiently small and choice of parameters as in Section \ref{s:parameters} we have:
\begin{enumerate}
\item $\P(\cD\mid \cE_4) \le e^{-\log^2(\frac{1}{\theta})},$
\item $\E((X^*-X_*)^2\mid \cE_4) {\le r^{2/3}}.$
\end{enumerate}
\end{lemma}

The proof of the above lemma is technical and will be postponed to Section \ref{s:l2}. 

\subsection{Construction of $\ce$}
We are now ready to define the event $\ce$. First we define certain nice events conditioned on $\cE_4$ towards the proof of Proposition \ref{p:event}.
  
\begin{enumerate}
\item Let $\ce_5$ denote the set of all $\omega=\omega_{\theta}\in \ce_4$ such that 
$$\E[(X^*-X_*)^2\mid \omega] \leq 10r^{2/3},$$ 
\item 
 Let $\ce_6$ denote the set of all $\omega=\omega_{\theta}\in \ce_4$ such that 
$$\E[(Z+W-Y-X_*)^2 \mid \omega] \leq 4 \phi_0^2 C_0 r^{2/3},$$ for some universal constant $C_0.$
\item Let $\ce_7$ denote the set of all $\omega\in \ce_4$ such that 
$$ {\rm {Var}}~(X_*\mid \omega) \geq c_0\theta^{-1/2} r^{2/3}.$$
\end{enumerate}

We shall set 
\begin{equation}
\label{e:edef}
\ce:=\ce_1\cap \ce_5 \cap \ce_6\cap \ce_{7}.
\end{equation}

\subsection{Proof of Proposition \ref{p:event}}
It remains to prove Proposition \ref{p:event} using the $\ce$ defined above. First we need to establish some properties of $\ce$. 

\begin{proposition}
\label{p:positive}
There exists  $\beta>0$ depending on all parameters such that $\P(\ce)>\beta$.
\end{proposition}
The proof of this Proposition is deferred to Section \ref{s:positive}. 
We are now ready to prove Proposition \ref{p:event}.
\begin{proof}[Proof of Proposition \ref{p:event}]
Let $\ce$ be as defined above. By Proposition \ref{p:positive} we know that $\P(\ce)$ is bounded below as required. Fix $\omega=\omega_{\theta}\in \ce$. Observe that $Y$ is a deterministic function of $\omega$. Using linearity of covariance and Cauchy-Schwarz inequality, we have for each $\omega \in \ce$
\begin{eqnarray*}
{\rm {Cov}} (X, Z+W \mid \omega) &=& {\rm {Cov}} (X, Z+W-Y \mid \omega)\\
&=& {\rm {Cov}}(X_*, Z+W-Y\mid \omega)+ {\rm {Cov}}(X-X_*, Z+W-Y \mid \omega)\\
&=& {\rm {Var}} (X_*\mid \omega)+ {\rm {Cov}}(X_*, Z+W-Y-X_* \mid \omega)\\
&+& {\rm {Cov}}(X-X_*,X_*\mid \omega)+ {\rm {Cov}} (X-X_*,Z+W-Y-X_*\mid \omega)\\
&\geq & {\rm {Var}} (X_*\mid \omega)-\sqrt{{\rm {Var}}(X^*)}\left(\sqrt{{\rm {Var}}(X-X_*\mid \omega)}+\sqrt{{\rm {Var}}(Z+W-Y-X_*\mid \omega)}\right)\\
&-& \sqrt{{\rm {Var}}(X-X_*\mid \omega)}\sqrt{{\rm {Var}}(Z+W-Y-X_*\mid \omega)}.
\end{eqnarray*}

By definition of $\ce_5$, and the observation that $X^*\geq X \geq X_*$ we get that for each $\omega\in \ce$, ${\rm {Var}}(X-X_*\mid \omega)\leq 10r^{2/3}$. By definition of $\ce_6$, ${\rm {Var}}(Z+W-Y-X_*\mid \omega)=O(\phi_0^{2} r^{2/3}).$ The proof  is completed by the definition of $\ce_7$, observing that by our choices of parameters $\theta^{-1/2}\gg \phi_0^2$. 
\end{proof}  

We can now illustrate how the proof of Theorem \ref{t:large}, (ii) can be completed along the same lines. We shall only provide a sketch. 

\begin{proof}[Proof of Theorem \ref{t:large}, (ii)]
First observe that in the notation of the above proof, using Cauchy-Schwarz inequality we have for all $\omega\in \ce$, 
$${\rm Var}(Z+W\mid \omega) \geq {\rm Var}(X_*\mid \omega)-2\sqrt{{\rm Var}(Z+W-Y-X_*\mid \omega)}\sqrt{{\rm Var}(X_*\mid \omega)}.$$
By definition of $\ce_6$ and $\ce_7$, we get that for $\theta$ sufficiently small and for all $\omega\in \ce$, we have ${\rm Var}(Z+W\mid \omega)\geq c(\theta) r^{2/3}$ for some $c(\theta)>0$. Now we make the same definitions as before, but interchange the roles of $r$ and $(n-r)$. Let the event corresponding to $\ce$ be now denoted $\ce'$. The analogue of Proposition \ref{p:positive} and the above observation now implies that for $1\ll n-r \ll n$ there exists a positive probability set $\ce'$ such that for each $\omega\in \ce'$, ${\rm Var}(T_{n}\mid \omega)\geq c(n-r)^{2/3}$ (and $T_{r}$ is a deterministic function of $\omega\in \ce'$). This implies for some constant $c'>0$ we have 
$$c'(n-r)^{2/3} \le \inf_{\lambda}{\rm Var}(T_n-\lambda T_r)= 
 (1-\mbox{Corr}^2(T_r,T_n)){\rm {Var}}(T_n);$$
 which completes the proof.
\end{proof}

The remainder of the paper is devoted to the proof of Proposition \ref{p:positive}.

\section{Probability of favourable geometric events}
\label{s:positive}

In this section we prove Proposition \ref{p:positive}. Recall the parameters from Section \ref{s:parameters}. We need a number of lemmas. We start by showing that $\ce_1$ is extremely likely. 
\begin{lemma}
\label{l:e1}
There exists positive constants  $r_0$, $\delta_1$ such that for all $\delta_1 n > r> r_0$, we have
$$\P(\ce_1)\geq 1-e^{-c\log ^{4} (\theta^{-1})}.$$
\end{lemma}

\begin{proof}
Recall the two events whose intersection $\ce_1$ consists of. That the first of those has probability at least $1-e^{-c\log ^{4} (\theta^{-1})}$ is an immediate consequence of Theorem \ref{t:aest}. The probability lower bound for the second event also follows from Theorem \ref{t:aest} after discretizing the line segment $\L_{r,\phi}$  into $\theta^{-100}$ intervals of $O(1)$ length and taking a union bound. 
\end{proof}

The next lemma shows that $\ce_4$ occurs with positive probability. 

\begin{lemma}
\label{l:e3}
There exists $\epsilon=\epsilon(\phi,L)>0$ such that $\P(\ce_4)>\epsilon$.
\end{lemma}

\begin{proof}
This proof is a minor variant of the proof of Lemma 8.3 of \cite{BSS14}. We omit the details. 
\end{proof}

The next lemmas will show that conditional on $\ce_4$, the events  $\ce_5$, $\ce_6$ and $\ce_7$ are quite likely. The proof for $\ce_5$ will already follow from Lemma \ref{geodimp} (2). So we focus on $\ce_6.$

\begin{lemma}
\label{l:e6}
There exists positive constants  $r_0$, $\delta_1$ such that for all $\delta_1 n > r> r_0$, we have
$$\E[(Z+W-Y-X_*)^2  \mid \ce_4] \leq 4\phi_0^2 r^{2/3}.$$
\end{lemma}

\begin{proof}
Let $A$ be the intersection of all the nice events $\cE_1, \cD^c, (\cB^{\phi_1})^c, (\cA_{\rm loc}^{\phi_1})^{c}, \cC^{c}$ along with the event ${\sT}:= X_{\theta}\ge 4r-\frac{1}{\theta^{2}}r^{1/3}$. Thus 
\begin{align*}\E[(Z+W-Y-X_*)^2  \mid \ce_4]&=\E[(Z+W-Y-X_*)^2 1_A \mid \ce_4]+\E[(Z+W-Y-X_*)^2  1_{A^c}\mid \ce_4]\\
&\le \E[(Z+W-Y-X_*)^2 1_A \mid \ce_4]+\E[(Z+W-Y-X_{\theta})^2  1_{A^c}\mid \ce_4]
\end{align*}

We first bound the first term. Note that on $(\cA_{\rm loc}^{\phi_1})^c$ the vertex $v\in \L_{r,\phi_1}$. Moreover we claim that  on $\cC^c \cap \cE_1\cap \cD^{c}\cap {\sT}$ in fact  $v\in \L_{r,\phi_0}.$
To see this consider any path $\gamma$ from $\bf{0}$ to $w=(w_1,w_2),$ followed by $\Gamma_{w,n}$ with weights $Z'$ and $W'$ respectively. Assume $w$ lies in $\L_{r,\phi_1}\setminus \L_{r,\phi_0}.$
We analyze two cases: 
\begin{enumerate}
\item $\gamma$ exits $R_{\phi}:$ Since we are on the event $\cC^c\cap \ce_1,$ this implies $Z'+W'-Y\le 4r-\frac{1}{\theta^4}r^{1/3}$ which is less than $X_*$ since we are also on the event  ${\sT}.$
\item $\gamma$ does not exit $R_{\phi}:$ Here the definition of the event $\cD$ implies that 
$$\ell(\gamma)-X_*\le -\frac{|w_1-w_2|^2}{100 r \theta^{3/2}}.$$ Whereas $\ce_1$ implies that $|W'-Y|\le \sqrt{|w_1-w_2|}\log^{9}(\theta^{-1}).$ Since $|w_1-w_2|^2\ge r\sqrt{|w_1-w_2|},$ putting these together we get $X_*+Y\ge Z'+W'.$
\end{enumerate}

Thus 
\begin{align*}
\E[(Z+W-Y-X_*)^2 1_A \mid \ce_4] &\le \E\left[\sup_{w \in \L_{r,\phi_0}}|T_{w,{\bf{n}}}-Y|^2\mid \cE_4\right]+\E[(X^*-X_*)^2\mid \cE_4]\\
&=O(\phi_0 r^{2/3})
\end{align*}
where in the last inequality we use that the first term is independent of the conditioning, Theorem \ref{t:aest} and Lemma \ref{geodimp}.
Now to bound $\E[(Z+W-Y-X_{\theta})^2  1_{A^c}\mid \ce_4]$ we use Cauchy-Schwarz inequality along with the fact that $\P(A\mid \ce_4)\ge 1-e^{-\log^2(\theta)}$ and that by FKG inequality $$\E[(Z+W-Y-X_{\theta})^4\mid \ce_4]\le \E[(Z+W-Y-X_{\theta})^4]$$ along with the following bound:
 \begin{align*}E[(Z+W-Y-X_\theta)^4]
&\le \E[(X^*+W-Y-X_\theta)^4]\\
&\le O(\E[(X^*-X_{\theta})^4]+\E[(W-Y)^4])\\
&=O(\theta^{-4}r^{4/3}).
\end{align*}
For the last equality we use Lemma \ref{l:4th1} and observe that $\E[(W-Y)^4]=O(r^{4/3})$ can be deduced from Lemma \ref{l:yw} and Lemma \ref{l:p2l} as in the proof of Proposition \ref{p:easyu}.
\end{proof}

It remains to control the conditional probability of $\ce_7$. Before stating a lemma towards that we need to make the following notations, which will be useful throughout the rest of the paper. We divide the rectangles $R_{\theta}$ (and $R_{2\theta}$) by parallel lines $\mathcal{L}_{i}:=\{x+y=2i\theta^{3/2}r\}$. The part of $R_{\theta}$ (resp.\ $R_{2\theta}$) enclosed between the lines $\mathcal{L}_{i}$ and $\sL_{i+1}$ shall be denoted by $R_{\theta}^{i}$ (resp.\ $R_{2\theta}^{i}$). For $\omega\in \Omega_{\theta}$, an index $i\in [\frac{1}{5}\theta^{-3/2}]$ is called \textbf{good} if 
$$\P\left(\sup_{u\in R_{2\theta}^{5i}, v\in R_{2\theta}^{5i+4}} T_{u,v} - 2|d(u)-d(v)|\leq  C^*\theta^{1/2}r^{1/3}\mid \omega \right) \geq 0.99$$
for some fixed large constant $C^*$. Let $\ce_{8}$ denote the subset of all $\omega\in \ce_4$ such that the fraction of good indices is at least a half. We now relate $\ce_{7}$ to $\ce_8$. 

\begin{lemma}
\label{l:realvar}
For each $C^*$, there exists $c_0>0$ sufficiently small (in the definition of $\ce_7$) such that $\ce_{8}\subseteq \ce_{7}$.
\end{lemma}

Proof of Lemma \ref{l:realvar} is rather involved and is provided in Section \ref{s:var}. The next lemma deals with the conditional probability of $\ce_8$.

\begin{lemma}
\label{l:e8}
If $C^*$ is sufficiently large, and $\theta$ is sufficiently small, then $\P(\ce_8\mid \ce_4)\geq 0.9$.
\end{lemma}

\begin{proof}
Observe that the number of good indices is negative in the configuration of weights outside $R_{\theta}$. Let $N$ denote the fraction of bad indices, and it follows by the FKG inequality that $\E[N\mid \ce_4]\leq \E N \leq 0.001$ for $C^*$ sufficiently large where the final inequality follows from Proposition \ref{t:treesup}. The proof of the lemma is completed using Markov inequality.
\end{proof}

We are now ready to prove Proposition \ref{p:positive}.
\begin{proof}[Proof of Proposition \ref{p:positive}]
Observe that by Markov inequality we have $\P(\ce_5\mid \ce_4)\geq 0.9$ and $\P(\ce_6\mid \ce_4)\geq 0.9$ using Lemma \ref{geodimp} (2) and Lemma \ref{l:e6}. Using Lemma \ref{l:realvar} and Lemma \ref{l:e8} we get $\P(\ce_8\mid \ce_4)\geq 0.9$. Invoking Lemma \ref{l:e1} and the fact that $\ce_1$ is independent of $\ce_4$ implies. $\P(\ce\mid \ce_4)\geq 1/2$. The proof of the proposition is completed by using Lemma \ref{l:e3}.
\end{proof}

It remains to prove Lemma \ref{geodimp} and Lemma \ref{l:realvar}. These are proved in the next two sections respectively.

\section{Barriers force localization}
\label{s:l2}
In this section we prove Lemma \ref{geodimp}. We need some new notation. By definition $R_{\phi}$ denotes the union of $R_{\theta}$, $U_1$ and $U_2$. Let $u_0, u_1, \ldots, u_{L}$ and $v_0, v_1,\ldots, v_{L}$ denote points on the two longer boundaries of $R_{\theta}$ equally spaced at distance $\frac{r}{L}$. Let the sub-rectangles defined by them be labeled $S_1,S_2,\ldots S_{L},$ (see Figure \ref{barrier1}).  
\begin{figure}[h]
\includegraphics[width=.90\textwidth]{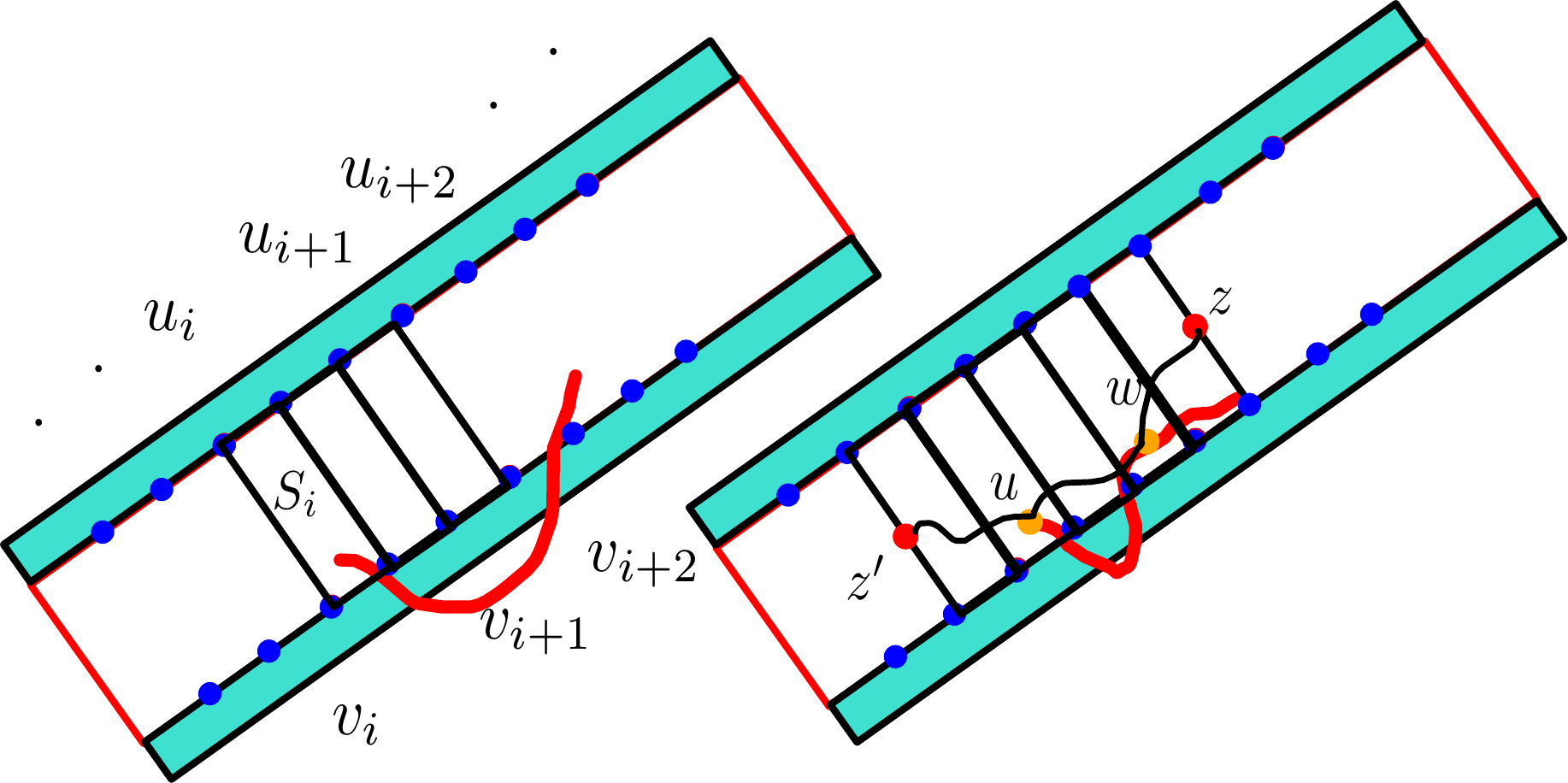} 
\caption{{Figure illustrating the proofs of Lemma \ref{l:123} that the events ${\rm{Int}}$, ${\rm{NoTF}}$ hold with high probability  conditioned on the barrier event $\ce_4.$ The blue regions denote the rectangles $U_1$ and $U_2$ in the definition of $\ce_4.$ The two figures consider the possibilities of the path $\Gamma^*$ wandering outside $R_{2\theta}:$ (a) having a long excursion in the barrier region, (b) causing high transversal fluctuation for polymers between consecutive visits to the $S_i's$. The second case is ruled out by comparing the length of the path with high transversal fluctuation with the path that stays inside the strip and showing that the latter is longer contradicting that the fluctuating path is in fact a geodesic.}}
\label{barrier1}
\end{figure}
 Also for any rectangle $S$, let us denote by $\mathcal{S}(S)$ all pairs of points $(u,v)\in S^2$ such that $u\preceq v$ and the slope of the line joining $u$ and $v$ is between $1/2$ and $2$. 
 
Let us now define the event which essentially ensures that all the lengths of the polymers constrained to lie in $R_{\theta}$ are reasonably typical i.e., none of them is too small. Note that this is independent of $\cE_4.$
\begin{align*}
\cA_1&:=\{T_{u,v}^{R_{\theta}}-\E T_{u,v} \geq -\theta^{-1}\log^{10}\theta^{-1} r^{1/3}\,\, \forall (u,v)\in \mathcal{S}(R_{\theta})\},
\end{align*}
We now define the following desirable events for the path $\Gamma^*$. 
The first one ensures that the polymer intersects all the $S_i$ until very close to the line $\L_r.$
\begin{equation}\label{def213}
{\rm{Int}}:=\left\{\Gamma^* \text{ intersects } S_i \text{ for all } i\ge 1 \text { such that } S_i \subset \{x+y \le 2r-2\theta^{3/2}r\}  \right\}.
\end{equation}
For convenience let the maximum $i$ satisfying the above constraint be denoted $i_0.$
The next one ensures that $\Gamma^*$ does not venture out of $R_{2\theta}$ until after going beyond the line $\{x+y \le 2r-2\theta^{3/2}r\}$
\begin{equation}\label{def214}
{\rm{NoTF}}:=\left\{\Gamma^* \text{ does not intersect } \{x+y \le 2r-2\theta^{3/2}r\} \cap R_{2\theta}^c \right\}.
\end{equation}

We now have the following lemma. 

\begin{lemma} $\P(\cA_1\mid \ce_4)=\P(\cA_1)\ge 1-e^{\log^{2}(\frac{1}{\theta})}.$
\end{lemma}

\begin{proof}This follows from Proposition \ref{t:treebox} and independence of $\cA_1$ and $\cE_4.$
\end{proof}

\begin{lemma}
\label{l:123} We have the two following probability bounds:
\begin{enumerate}
\item $\P({\rm Int}^c\cap \cA_1\mid \ce_4)\le e^{-\frac{1}{\theta^2}}.$
\item $\P({\rm NoTF}^c\cap {\rm Int}\cap \cA_1\mid \ce_4)\le e^{-\frac{1}{\theta^2}}.$
\end{enumerate}
\end{lemma}

\begin{proof}
We first prove (1).
By Lemmas \ref{macrolocal} and \ref{TFlocal} we can assume that the complements of the events $\cB^{\phi_1}$ and $\cC_{\phi_1}$ hold, which in particular imply that $\Gamma^*$ stays constrained within $R_{\phi}.$ 
Now on the event ${\rm Int}^c$ there exists vertices $u, w$ on the polymer $\Gamma^*$ such that $u \in R_\theta$ and $w$ is the next point on $\Gamma^*\cap R_{\theta}$ if such a point exists or $w=v^*$ which is the end point of $\Gamma^*.$ Moreover $|u-w|\ge \frac{r}{L}$ and  $\Gamma_{u,w}$ lies entirely outside $R_{\theta}.$
Now by definition of $\cA_1$ and $\cE_4$ it follows that in the former case (when $w\neq v^*$) $$T^{R_{\theta}}_{u,w}\ge \ell (\Gamma^*_{u,w})$$ where $\Gamma^*_{u,w}$ denotes the segment of $\Gamma^*$ between $u$ and $w$ (which is the same as $\Gamma_{u,w}$) which is a contradiction. When $w=v^*$ a similar consideration shows $T^{R_{\theta}}_{u,{\bf{r}}}\ge \ell (\Gamma^*_{u,w})$ again leading to a contradiction.

We now prove (2). Now by definition on $\rm {Int},$ the path  $\Gamma^*$ intersects $S_i$ for all $i\le i_0$ (where $i_0$ appears in \eqref{def213}.)
Thus on the complement of $\rm{NoTF},$ there  must exist $i\le i_0$ and vertices $u$ and $w$ on $\Gamma^*$ where $u\in S_i, w\in S_{i+2}$ and the geodesic $\Gamma_{u,w}$ which is a segment of $\Gamma^*$ exits $R_{2\theta}$ and hence has unusual transversal fluctuation. We will now invoke Proposition \ref{p:tfnew} to rule this out. However note that Proposition \ref{p:tfnew} is for a fixed rectangle and considers polymers going from one side to the other side whereas in our case the points $u,w$ lie somewhere in the rectangles $S_{i}$ and $S_{i+2}.$ 
Thus to apply Proposition \ref{p:tfnew} consider the new rectangle $S_*=S_{i-1}\cup S_{i}\cup S_{i+1} \cup S_{i+2} \cup S_{i+3}.$
Let $z'$ be the midpoint of the bottom  side of $S_{i-1}$ and 
 $z$ be the midpoint of the top side of $S_{i+3},$ (see Figure \ref{barrier1}).
Now consider the path $\gamma$ obtained by concatenation of  $\Gamma^{R_{\theta}}_{z',u},\Gamma^{R_{\theta}}_{u,w}, \Gamma^{R_{\theta}}_{w,z}.$
Similarly consider the path $\gamma'$ obtained by concatenation of  $\Gamma^{R_{\theta}}_{z',u},\Gamma^{*}_{u,w}, \Gamma^{R_{\theta}}_{w,z}$, where $\Gamma^*_{u,w}$ is the restriction of $\Gamma^*$ between $u$ and $w$. Since $\Gamma^{*}_{u,w}$ is in fact $\Gamma_{u,w},$ it follows that $\ell(\gamma')\ge \ell(\gamma)$. Now by Proposition \ref{p:tfnew} and our choice of parameters except for an event of probability at most $e^{-\frac{1}{\theta^2}}$, uniformly over all $i$ as above, $\ell(\gamma')\le 4 (\frac{5r}{L})-\theta^{-20}r^{1/3}.$
Moreover, note that by choice of the points $z',u,w,z,$ the pairs $(z',u)$, $(u,w)$, and $(w,z)$ are all in $\sS(S_*)$. Observe also that using Theorem \ref{t:moddev} we have $\E T_{w,u}+ \E T_{u,v}+ \E T_{v,z}\geq 4(\frac{5r}{L})-cr^{1/3}$ for some constant $c>0$. Thus the event $\cA_1$ implies that  
$\ell(\gamma)\ge 4 (\frac{5r}{L})-\theta^{-2}r^{1/3}$ contradicting  $\ell(\gamma')\ge \ell(\gamma)$ with failure probability at most $e^{-\frac{1}{\theta^2}}$. This completes the proof. 
\end{proof}

We are now ready to prove Lemma \ref{geodimp}.
\begin{proof}[Proof of Lemma \ref{geodimp}] We first prove (1). Let $\gamma$ from $\mathbf{0}$ to $w\in \L_{r,\phi_1}\setminus \L_{r,\phi_0}$ denote the path attaining the supremum in the definition of the event $\cD$. The same argument as above shows that $\gamma$ intersects every $S_i$ for all $i\le i_0$ and does not exit $R_{{2\theta}}$ until after intersecting $S_{i_0}$ (say at $v_0$) with very small failure probability. Note that on this event $$\ell(\gamma)-X_*\le T_{v_0,w}-T_{v_0,\mathbf{r}}^{R_{\theta}}\le \sup_{v_0\in S_{i_0}, w \in \L_{r,\phi_1}\setminus \L_{r,\phi_0}}[T_{v_0,w}-T_{v_0,\mathbf{r}}^{R_{\theta}}].$$ The last term is now bounded by the FKG inequality, Propositions \ref{t:treesup} and \ref{t:treebox} we are done, by noticing that $\E T_{v_0,w}\leq  2(2r-d(v_0))-\frac{(w_1-w_2)^2}{50\theta^{3/2}r}$.

(2) follows by  a similar argument. Consider the nice event 
$$A:=\cE_1 \cap \cD^c \cap  (\cB^{\phi_1})^c \cap (\cA_{\rm loc}^{\phi_1})^{c} \cap  \cC^{c}\cap  \cA_1 \cap {\rm{Int}} \cap {\rm{NoTF}}$$
with super polynomially small (in $\theta$) failure probability conditional on $\ce_4$. On $A$, let $v_0$ be the point in $S_{i_0}$ as in the previous part. Let $T^*_{v_0}$ denote the length of the best path from $v_0$ to the line $x+y=2r$. It follows that, on $A$, $X^*-X_*\leq T^*_{v_{0}}-T_{v_{0},\mathbf{r}}^{R_{\theta}}$. 
Thus $$\E[(X^*-X_*)^2\mid \ce_4]=\E[(X^*-X_*)^2 1_{A}\mid \ce_4]+\E[(X^*-X_*)^2 1_{A^c}\mid \ce_4].$$ The first one is dominated by $\E[\sup_{v\in S_{i_0}}(T^*_{v_0}-T_{v_{0},\mathbf{r}}^{R_{\theta}})^2\mid \ce_4]$ which we show to be $O(\theta r^{2/3})$ by Lemma \ref{l:sideregularity}, and the FKG inequality. The second one is bounded by a simple application of Cauchy-Schwarz inequality, the super polynomial bound ($e^{-\log^2{\theta}}$) of $\P(A^c\mid \ce_{4})$, and a bound on the fourth moment of $|X^*-X_*|$ which is proved in the next lemma.  
\end{proof}

\begin{lemma}
\label{l:4th1}
We have for $\theta$ sufficiently small 
$$ \E[ (X^*-X_{*})^4 \mid \ce_4] = O(\theta^{-4} r^{4/3}).$$
\end{lemma}

\begin{proof}
Observe that by the FKG inequality it suffices to prove 
$$\E[X^*-X_{\theta}]^4 = O(\theta^{-4} r^{4/3}).$$
This is done by controlling $\E[X^*-X]^4$ by Lemma \ref{l:p2l},  and observing that $\E[X-X_{\theta}]^4= O(\E[X-4r]^4+\E[X_{\theta}-4r]^4)$ and using Theorem \ref{t:moddev} and Lemma \ref{l:4th2}.
\end{proof}

\section{Variance lower bounds for constrained polymer}
\label{s:var}

It remains to prove Lemma \ref{l:realvar}. Clearly it immediately follows from the following proposition.

\begin{proposition}
\label{p:variance}
There exists $c_0>0$ depending on $C^*$ such that for each $\omega\in \ce_8$, we have ${\rm Var}~(X_*\mid \omega)\geq c_0\theta^{-1/2} r^{2/3}$. 
\end{proposition}

We shall prove Proposition \ref{p:variance} by a standard Doob Martingale variance decomposition. Let us first fix $\omega\in \ce_8$. Let $J=\{i_1<i_2<\cdots < i_{|J|}\}$ be an enumeration of good indices (observe that $J$ is deterministic given $\omega$) i.e., for each $j\in J$
$$\P\left(\sup_{u\in R_{2\theta}^{5j}, v\in R_{2\theta}^{5j+4}} T_{u,v}- 2|d(u)-d(v)| \leq  C^*\theta^{1/2}r^{1/3}\mid \omega\right) \geq 0.99$$ 
By definition of $\ce_{8}$, we know that $|J|\geq \frac{1}{10}\theta^{-3/2}$. Let us now define a sequence of $\sigma$-fields, $\cg_0 \subseteq \cg_1 \subseteq \cdots \subseteq \cg_{|J|}$. Where $\cg_0$ is generated by the configuration $\omega$ on $\Z^2\setminus R_{\theta}$ together with the configuration on $R_{\theta}^{i}$ for all $i$ not of the form $5j+2$ for some $j\in J$, and for $j\geq 0$, $\cg_{j+1}$ is the sigma algebra generated by $\cg_{j}$ and the configuration on $R_{\theta}^{5i_{j+1}+2}$. 

We shall consider the Doob Martingale $M_j:=\E[X_*\mid \cg_{j}]$. Observe that by the standard variance decomposition of a Doob Martingale it follows that 

\begin{equation}
\label{l:vardecom}
{\rm Var}~(X_*\mid \omega)\geq \sum_{j=1}^{J} \E (M_j-M_{j-1})^2.
\end{equation}

Clearly it suffices to prove the next lemma.

\begin{lemma}
\label{l:eachgood}
There exists $c>0$ such that for each $i_j$ in $J$, there is a subset $A_j$ with probability at least $1/2$ measurable with respect to $\cg_{j-1}$ such that such that on $A_j$ we have
$$\E[(M_j-M_{j-1})^2\mid \cg_{j-1}] \geq c\theta r^{2/3}.$$ 
\end{lemma}
\begin{figure}[htbp!]
\includegraphics[width=.80\textwidth]{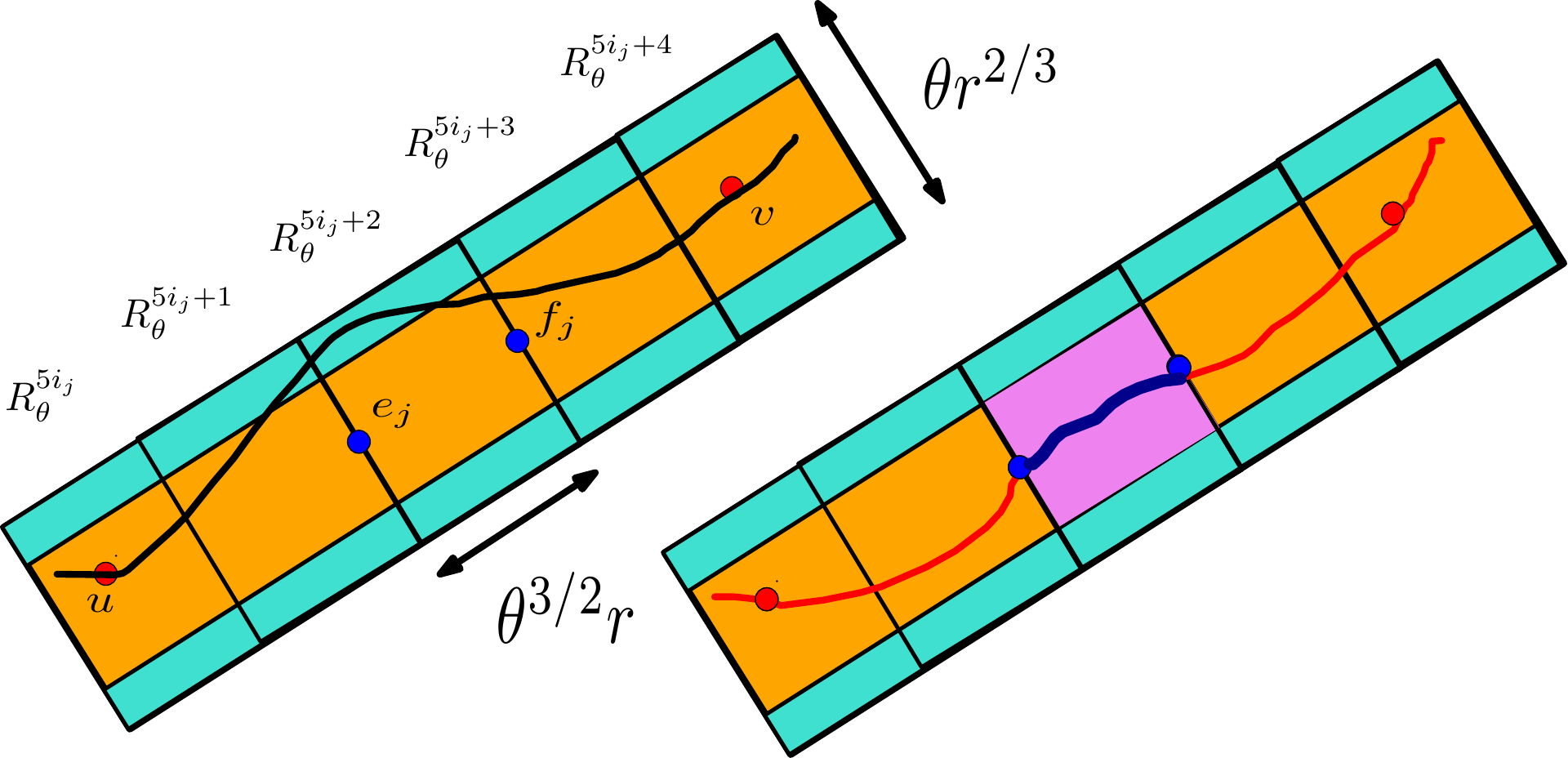} 
\caption{Figure illustrating the proof of Lemma \ref{l:eachgood}. The picture on the left depicts a geodesic locally and the picture on the right shows the same environment with the middle block $R^{i}_{\theta}$ resampled to have a high point to point polymer weight. This new fertile environment can now be used to reroute the geodesic to get an on-scale variance of $\Theta(\theta r^{2/3}).$ This argument is repeated at $O(\theta^{-3/2})$ blocks along with a Doob Martingale argument obtaining a lower bound of $\theta^{-1/2}r^{2/3}.$}
\label{fig.var1}
\end{figure}

Before starting to prove this lemma we need to introduce some more notation. Fix $i_j\in J$. Let $e_j$ and $f_j$ denote the midpoints of the left and right side of $R_{\theta}^{5i_j+2}$. Let $\omega_1$ denote a configuration on $R_{\theta}^{5i_j+2}$ drawn from the i.i.d.\ Exponential distribution. Let $\tilde{\omega}_1$ denote a a configuration on $R_{\theta}^{5i_j+2}$ drawn from the i.i.d.\ Exponential distribution conditional on $T^{R_{\theta}}_{e_j,f_j}\geq 4\theta^{3/2}r+ 2C^*\theta^{1/2}r^{1/3}$ (denote this event by $F$). Let $\mu$ denote a coupling of $(\omega_1,\tilde{\omega}_1)$ such that $\tilde{\omega}_1 \geq \omega_1$ point wise (such a coupling exists by the FKG inequality). Let $\omega_0$ denote a configuration on all the vertices that are revealed in $\cg_{j-1}$. Let $\omega_2$ denote a configuration on $\cup_{\ell >j} R_{\theta}^{5i_{\ell}+2}$. Let $\underline{\omega}= (\omega_0, \omega_1, \omega_2)$ and $\underline{\tilde{\omega}}= (\omega_0, \tilde{\omega}_1, \omega_2)$ denote the two environments and let $X^*(\underline{\omega})$ (resp.\ $X^*(\underline{\tilde{\omega}})$) denote the value of the statistic $X_*$ computed in the environment $\underline{\omega}$ (resp.\ $\underline{\tilde{\omega}}$). Observe now that for $(\omega_1, \omega_2)=\mu(\omega_1, \omega_2)$ for the coupling $\mu$ described above we have 

$$(\E[M_j\mid F, \cg_{j-1}]-M_{j-1})(\omega_0)=\int (X^*(\underline{\tilde{\omega}})-X^*(\underline{\omega}))~ d\mu~d\omega_2.$$

Observing that $X^*(\underline{\tilde{\omega}})$ is point wise larger than or equal to $X^*(\underline{\omega})$ and that $F$ is independent of $\cg_{j-1}$ and that 
$$\E[(M_j-M_{j-1})^2\mid \cg_{j-1}] \geq  \P(F)(\E[M_j\mid F, \cg_{j-1}]-M_{j-1})^{2},$$
Lemma \ref{l:eachgood} is implied by the following two lemmas.

\begin{lemma}
\label{l:fbound}
For each $C_*>0$, there exists $c>0$ such that
$$\P(T^{R_{\theta}}_{e_j,f_j}> 4\theta^{3/2}r+ 2C^* \theta^{3/2}r) \geq c.$$
\end{lemma}

\begin{proof}
Follows from Lemma \ref{l:uppertail}.
\end{proof}

\begin{lemma}
For $C_*$ sufficiently large, there is a set $A_j$ of configurations $\omega_0$ with probability at least $0.9$ and a set $B_j$ of configurations $\omega_2$ with probability at least $0.9$ such that for each $\omega_0\in A_j$ and $\omega_2\in B_j$ there exists a set $C_j$ of configurations $\omega_1$ with $\mu$-probability at least $0.9$ such that for $\underline{\omega}$ coming from these sets (and the coupled $\underline{\tilde{\omega}}$) we have 
$X^*(\underline{\tilde{\omega}})-X^*(\underline{\omega})\geq \frac{C^*}{10}\theta^{1/2}r^{1/3}$.
\end{lemma}

\begin{proof}
Let us first define $A_j$, $B_j$ and $C_j$ and we shall afterwards verify that they satisfy the conclusions of the lemma. Notice that $\omega_0$ and $\omega_2$ both refer to certain different  rectangles $R_{\theta}^{j}$. Let us call these disjoint sets of indices $J_1$ and $J_2$. Let  $\tilde{A}_j$ denote the event that for each $k\in J_1$ we have 
$$ T_{u,v}^{R^k_{\theta}} \geq \E T_{u,v}- \log ^{10}(\theta^{-1}) \theta^{1/2} r^{1/3}$$
for all $(u,v)\in \sS(R_{\theta}^{k})$. Further, we also ask for $k=5i_j, 5i_j+1, 5i_j+3, 5i_j+4$ a stronger lower bound $\E T_{u,v}-\frac{C^*}{4}\theta^{1/2} r^{1/3}$ holds. Clearly for $C^*$ sufficiently large we have $\P(\tilde{A}_j)\geq 0.999$ using Proposition \ref{t:treebox}. Let $B_i$ denote the event that for each $k\in J_2$ we have 
$$ T_{u,v}^{R^k_{\theta}} \geq \E T_{u,v}- \log ^{10}(\theta^{-1}) \theta^{1/2} r^{1/3}$$
for all $(u,v)\in \sS(R_{\theta}^{k})$. Again, by Proposition \ref{t:treebox} we have $\P(B_j)\geq 0.9$. Finally let $\tilde{C}_j$ denote the set of all configurations $\omega_1$ such that $$ T_{u,v}^{R^{5i_j+2}_{\theta}} \geq \E T_{u,v}- \log ^{10}(\theta^{-1}) \theta^{1/2} r^{1/3}$$
$(u,v)\in \sS(R_{\theta}^{5i_j+2})$. Again we have $\P(\tilde{C}_j)\geq 0.999$. {It follows by Fubini's theorem and the fact that $j$ is a good index that there is a subset $A_j$ of $\tilde{A}_j$ with probability at least $0.9$ such that for each $\omega_0\in A_j$, there exists a subset $C_j$ of $\tilde{C}_j$ such that for all $\omega_1\in C_j$ we have $T_{u,v}(\underline{\omega})\leq 2|d(u)-d(v)|+C^*\theta^{1/2} r^{1/3}$ for each $u\in R_{2\theta}^{5i_j}$ and $v\in R_{2\theta}^{5i_j+4}$}.
It remains to prove that for $\underline{\omega}$ constructed from $A_j,B_j,C_j$ as described in the statement of the proposition and the coupled $\underline{\tilde{\omega}}$) we have 
$X^*(\underline{\tilde{\omega}})-X^*(\underline{\omega})\geq \frac{C^*}{2}\theta^{1/2}r^{1/3}$. Let $\gamma_1$ denote the path attaining $X_*$ in the environment $\underline{\omega}$. Observe that, by definition of $A_j,B_j$ and $C_j$ we conclude as in the proof of Lemma \ref{l:123} that $\gamma_1$ must intersect $R_{\theta}^{5i_j}$ and $R_{\theta}^{5i_j+4}$. say at points $u$ and $v$. Observe that by definition 
$T_{u,v}^{R_{2\theta}}(\underline{\omega})\leq 2|d(u)-d(v)|+C^*\theta^{1/2} r^{1/3}$. On the other hand 
$$T_{u,e_j}(\underline{\tilde{\omega}})+T_{e_j,f_j}(\underline{\tilde{\omega}})+T_{f_j,v}(\underline{\tilde{\omega}}) \geq 2|d(u)-d(v)|+\frac{11C^*}{10}\theta^{1/2} r^{1/3}$$
concluding the proof. 
\end{proof}
 
Before concluding we note that our proof of Proposition \ref{p:variance} can be used to obtain the order of variance of the length of the best path constrained to stay within a thin cylinder, answering a question raised in \cite{DPM17}. 

\begin{proposition}
\label{l:variancelb}
There exists absolute positive constants $C_1$ and $C_2$ and $r_0>0$ such that for all $r>r_0$ and $\theta \leq 1$ we have
$$C_1\theta^{-1/2} r^{2/3} \leq {\rm {Var}}~X_{\theta} \leq C_2 \theta^{-1/2} r^{2/3}.$$
\end{proposition}

The lower bound in the above proposition can be proved using a simpler version of the argument used in the proof of Proposition \ref{p:variance}. Upper bound follows from a Poincar\'e inequality argument after revealing one box at a time as before.

\bibliography{slowbond}
\bibliographystyle{plain}

\appendix
\section{}
\label{s:appa}
For easy reference purpose we record in this section the results from \cite{BSS14} that we have used throughout. All these results use the following basic ingredient which is implicit in \cite{BFP12}, as explained in \cite{BSS14}.

\begin{theorem}
\label{t:moddev}
Then  there exist constants $N_0,t_0,c>0$ such that we have for all $n>N_0, t>t_0$ and all $h \in (\frac{1}{2},2)$
$$\P[|T_{\mathbf{0}, (n,hn)}- n(1+\sqrt{h})^{2}|\geq tn^{1/3}]\leq e^{-ct}.$$
\end{theorem}

The following propositions are quoted from \cite{BSS14}, where the same were proved for Poissonian LPP.  As already explained in \cite{BSS14} the same proof goes through verbatim if we replace the moderate deviation estimates for Poissonian LPP, used there, by Theorem \ref{t:moddev}. We first need some notations. Consider the rectangle $U$ given by $0\leq x+y \leq 2r$, $|x-y|\leq r^{2/3}$. The next three propositions show that with large probability paths between all pairs of points in $U$ behave typically.

\begin{proposition}
\label{t:treeinf}
There exist absolute constants $c_1>0$, $r_0>0$ and $t_0>0$ such that we have for all $r>r_0$ and $t>t_0$
\begin{equation}
\label{e:treeinfgeneral}
\P\left(\inf_{(u,v)\in \mathcal{S}(U)} T_{u,v}-2|d(u)-d(v)| \leq -t r^{1/3}\right)\leq e^{-c_1t}.
\end{equation}
\end{proposition}

This is Proposition 10.1 of \cite{BSS14}.

\begin{proposition}
\label{t:treesup}
There exist absolute constants $c_1>0$, $r_0>0$ and $t_0>0$ such that we have for all $r>r_0$ and $t>t_0$
\begin{equation}
\label{e:treesupgeneral}
\P\left(\sup_{(u,v)\in \mathcal{S}(U)} T_{u,v}-\E T_{u,v} \geq t r^{1/3}\right)\leq e^{-c_1t}.
\end{equation}
\end{proposition}

This is Proposition 10.5 of \cite{BSS14}.

\begin{proposition}
\label{t:treebox}
There exists absolute constants $c_1>0$, $r_0>0$ and $t_0>0$ such that we have for all $r>r_0$ and $t>t_0$
\begin{equation}
\label{e:treeboxgeneral}
\P\left(\sup_{(u,v)\in \mathcal{S}(U)} T^{U}_{u,v}-2|d(u)-d(v)| \leq -t r^{1/3}\right)\leq e^{-c_1t^{1/3}}.
\end{equation}
\end{proposition}

This is Proposition 12.2 of \cite{BSS14}. 

We have chosen to state above the statements in their simplest forms, and these are the versions we need more often. However the results established in \cite{BSS14} were more general on several counts, let us now briefly explain the possible generalizations, as we shall also use them occasionally.

First, observe that we have done the centering in Proposition \ref{t:treesup} by $\E T_{u,v}$ whereas in the other results the centering is by $2|d(u)-d(v)|$. As a matter of fact the result holds with the centering $\E T_{u,v}$ in all the cases. It is easy to see using Theorem \ref{t:moddev} that for $(u,v)\in \sS(U)$, we have $\E T_{u,v}=2|d(u)-d(v)|-O(r^{1/3})$, so in the setting above the two centerings are equivalent. 

Second, all these results continue to hold for $r\times Cr^{2/3}$ rectangles also where $C$ is bounded away from $0$ and $\infty$, and also for rectangles with longer pair of sides parallel to some line with slope not necessarily equal to $1$, as long as the slope remains bounded away from $0$ and $\infty$ (in such cases one needs to invoke a stronger version of Theorem \ref{t:moddev} where the only constraint on $h$ is that it is bounded away from $0$ and $\infty$). In such cases, centering by the expectation gives a stronger result when the slope is away from $1$. We shall need to invoke Proposition \ref{t:treesup} in such settings, and therefore we chose to center by $\E T_{u,v}$ in the statement of that result.

Finally, we also do not need in the definition of $\sS(U)$, to restrict to the pairs of points so that the straight line joining them to have slope between $1/2$ and $2$. Again using the stronger version of Theorem \ref{t:moddev} we can include all pairs of points such that the slope of the straight line joining them has slope $(\frac{1}{\psi}, \psi)$ for some arbitrary large but finite constant $\psi$. 

The next proposition controls the transversal fluctuations of the geodesics.

\begin{proposition}
\label{t:tf}
Let $TF_{r}:= \sup\{|x-y|: (x,y)\in \Gamma_{r}\}$ denote the transversal fluctuation of tthe geodesic from $\mathbf{0}$ to $\mathbf{r}$. Then we have for all $k$ sufficiently large 
$$ \P(TF_{r}> kr^{2/3})\leq e^{-ck^2}.$$
\end{proposition}

This is Theorem 11.1 of \cite{BSS14} (see also Theorem 2 and Remark 1.5 of \cite{BSS17B}). Observe that this result also holds for $r\times Cr^{2/3}$ rectangles also where $C$ is bounded away from $0$ and $\infty$.

\section{}
\label{s:appb}
Here we collect a number of auxiliary results about point-to-line paths, paths constrained to be within an on scale rectangle, and paths with atypical transversal fluctuations. These are rather standard consequences of the results described in Appendix \ref{s:appa}, or their proofs in \cite{BSS14}, and hence we shall often give only sketches of the proofs.

 Let $U$ be as in the previous section. Also recall that $X^*$ denotes the weight of the polymer from $\mathbf{0}$ to the line $x+y=2r$, and $X=T_{r}$.
 
\subsection{Point-to-line polymer length cannot be too large}

\begin{lemma}
\label{l:p2l}
There exist positive constants $r_0,y_0$ and $c>0$ such that for all $r>r_0$ and $y>y_0$ we have
$$\P(X^*-X>yr^{1/3})\leq e^{-cy^{1/3}}.$$
\end{lemma}

\begin{proof}
Let $v_*$ denote the endpoint of the path attaining $X^*$. It is known (see Lemma 4.3 of \cite{FO17}, and also the proof of Lemma 11.3 in \cite{BSS14}) that $\P(r^{-2/3}|v_*-\mathbf{r}|\geq M)\leq e^{-cM}$ for $M$ sufficiently large. Clearly this together with Theorem \ref{t:aest} completes the proof of the Lemma. For completeness let us provide a brief sketch of the proof of the above statement. 

First notice that by Theorem \ref{t:moddev}, with failure probability $e^{-cM}$ the geodesic from $\mathbf{0}$ to $\mathbf{r}$ has length at least $4r-Mr^{1/3}$. So it suffices to show that the event that the length of the geodesic from $0$ to $v_{*}$ be at least $4r-Mr^{1/3}$ for some $v_*$ with $r^{-2/3}|v_*-\mathbf{r}|\geq M$ is exponentially unlikely in $M$. First observe that one can easily rule out the points $v_*$ where $|v_*-\mathbf{r}|$ is linear in $r$. For this, consider the line segment $\mathbf{L}'$ joining $(2r,0)$ to $(1.9 r, 0.1r)$. Clearly $\sup_{v\in \mathbf{L}'}T_{\mathbf{0},v} \leq T_{\mathbf{0}, (2r,0.1r)}$. It is easy to show using Theorem \ref{t:moddev} that is exponentially unlikely (in $M$) that $T_{\mathbf{0}, (2r,0.1r)}\geq 4r-Mr^{1/3}$. So it remains to consider the points $v_*$ such that $0.9 r^{1/3} \geq  r^{-2/3}|v_*-\mathbf{r}|\geq M$.

{Observe that, using Theorem \ref{t:moddev}, $\E T_{\mathbf{0},(r+xr^{2/3}, r-xr^{2/3})}=4r+O(r^{1/3})-cx^{2}r^{1/3}$, for some constant $c>0$. This shows that it is unlikely that $T_{\mathbf{0},(r+xr^{2/3}, r-xr^{2/3})}$ shall be competitive with $T_{r}$, for large $x$. We shall then take a union bound over sub-intervals of the line $x+y=2r$ of length $O(r^{2/3})$. The formal argument is similar to the proof of the second part of Lemma \ref{macrolocal} and we use the same notations as there. For $j$ large ( $j> M$), we get by using Theorem \ref{t:treesup} (applied to parallelograms with solpe bounded away from $0$ and $\infty$) that it is increasingly unlikely (probability at most $e^{-cM}f(j)$ where $f(j)$ is summable in $j$) that the best path from $\mathbf{0}$ to the interval-pair $\L_{r,j+1}\setminus \L_{r,j}$ will have length at least $4r-Mr^{1/3}$. Taking a union bound over values of $j$ upto order $\frac{1}{3} r^{1/3}$ gives the desired result.} 

\end{proof}

\subsection{Estimates for paths constrained in a rectangle}

\begin{lemma}
\label{l:uppertail}
There is $C,c>0$ such that for every $x>0$ we have
$$\P(T^{U}_{\mathbf{0},\mathbf{r}}>4r+xr^{1/3})\geq Ce^{-cx^{3/2}}.$$
\end{lemma}

\begin{proof}  
Clearly it suffices to prove the lemma for $x$ sufficiently large, fix such an $x$. Let $\rho>0$ be a fixed constant (to be chosen later depending on $x$ sufficiently small). Just for this proof let us denote $T^{\rho}_{n}$ to be the length of the best path from $\mathbf{0}$ to $\mathbf{n}$ constrained to be in the rectangle $0\leq x+y\leq 2n, |x-y|\leq \rho n^{2/3}$. Observe that for $\rho$ sufficiently small we have 
$$\liminf_{n\to \infty}\P[T_{\rho n}^{\rho^{-2/3}}\geq 4\rho n + (\rho n)^{1/3}] \geq \beta >0$$
where $\beta$ is independent of $\rho$. This follows from the fact that $\P[T_{\rho n}\geq 4\rho n + (\rho n)^{1/3}]$ is bounded below independent of $\rho$ (by Tracy-Widom scaling limit) and $\P[T_{\rho n} \neq T_{\rho n}^{\rho^{-2/3}}] \to 0$ as $\rho \to 0$. Now as $T_{n}^{1}$ stochastically dominates sum of $\frac{1}{\rho}$ many independent copies of $T_{\rho n}^{\rho^{-2/3}}$ it follows that 
$$\liminf_{n\to \infty}\P[T_{n}^{1} >4n+\rho^{-2/3}n^{1/3}] \geq \beta ^{1/\rho}.$$
The result follows by choosing $\rho=x^{-3/2}$.
\end{proof}

Observe also that the same result with the same proof holds if $U$ is replaced by an $r\times qr^{2/3}$ rectangle where $q$ is bounded away from $0$ and $\infty$. Recall the rectangle $R_{\theta}$. The next result bounds the moments of $T_{\mathbf{0}, \mathbf{r}}^{R_{\theta}}$.

\begin{lemma}
\label{l:4th2}
There exists $c>0$ such that for all $\theta$ and for all $r$ sufficiently large depending on $\theta$ we have 
$$ \E T_{\mathbf{0}, \mathbf{r}}^{R_{\theta}} \geq 4r-c\theta^{-1}r^{1/3}.$$
Further we have 
$$\E |T_{\mathbf{0}, \mathbf{r}}^{R_{\theta}} -4r|^{4} \leq C\theta^{-4}r^{4/3}.$$
\end{lemma}
\begin{proof}
Let $Z_{i}:=T^{R_{\theta}}_{\mathbf{i\theta^{3/2}r}, \mathbf{(i+1)\theta^{3/2}r}}$. First part of the lemma follows from noticing that $T_{\mathbf{0}, \mathbf{r}}^{R_{\theta}}\geq \sum Z_i$, and Proposition \ref{t:treebox}. For the second part of the lemma, observe that 
$$|T_{\mathbf{0}, \mathbf{r}}^{R_{\theta}} -4r|^{4} \leq  |\sum Z_i-4r|^{4}+ |T_{\mathbf{0},\mathbf{r}}-4r|^{4}.$$ 
That the expectation of the second term above is $O(r^{4/3})$ follows from Proposition \ref{t:treesup}. That the expectation of the first term is is $O(\theta^{-4}r^{4/3})$ is a consequence of the fact that $Z_{i}$'s are i.i.d.\ and Proposition \ref{t:treebox}.
\end{proof}

Finally we have the following lemma which compares the point-to-point distance constrained within a rectangle to the unconstrained point-to-side distance.

\begin{lemma}
\label{l:sideregularity}
Let $U'$ be the rectangle $\{0\leq x+y\leq r, |x-y|\leq r^{2/3}\}$. For $u$ in $U^*$, let $T^*_{u}$ denote the weight of the best path from $u$ to the line $x+y=2r$. Then we have for all $r$ sufficiently large,
$$ \E[\sup_{u\in U'} T^*_u-T_{u,\mathbf{r}}^{U}]^2=O(r^{2/3}).$$
\end{lemma}

\begin{proof}
The proof is done by showing separately $\E[\sup_{u\in U'} T^*_u -2(2r-d(u))]^2=O(r^{2/3})$ and $\E[\sup_{u\in U'} T_{u,\mathbf{r}}^{U}-2(2r-d(u))]^2=O(r^{2/3})$. The first part is done as in Lemma \ref{l:p2l} using the arguments of Proposition \ref{t:treesup}, and the second part follows from Proposition \ref{t:treebox}. We omit the details.
\end{proof}

\subsection{Bounds on paths with large transversal fluctuation}
From Theorem \ref{t:tf}, we know that any the polymer from $\mathbf{0}$ to $\mathbf{r}$ has transversal fluctuation $O(r^{2/3})$ with some probability. Even something stronger is true: any path with transversal fluctuation more than a large constant $h$ times $r^{2/3}$ is with high probability be smaller than $4r-f(h)r^{1/3}$ where $f(h)$ can be made arbitrarily large by taking $h$ large. We prove the following more general result. 

\begin{proposition}
\label{p:tfnew}
For any positive $\rho$, let Let $U_{\rho}=U_{\rho, r}$ be the rectangle $\{(x,y): 0\leq x+y \leq 2r, |x-y|\leq \rho r^{2/3}\}$. Let $C$ and $\alpha$ be fixed large positive constants. Let $A$ and $B$ be the bottom and top sides of $U_{C}$. For $C_2>0$
$$T^*=\sup\{\ell(\gamma): \gamma\in \Gamma_{A,B,C_2}\}$$
where $\Gamma_{A,B,C_2}$ denotes the set of all paths from some $u\in A$ to some $v\in B$ that exists $U_{C_2}$. Then  there exists $\beta>0$ (depending on $C,\alpha$) such that for all sufficiently large $r$ we have
$$\P(T^*\geq 4r-C_1r^{1/3}) \leq e^{-cC_2^{1/3}}$$
for some constant $C$ where $C_1=C^{\alpha}$ and $C_2=C^{\beta}$.  
\end{proposition}

It will follow from the proof that it suffices to take $C_2$ to be a large constant times $(C+C_1)^2$.

The proof of this proposition is essentially the same as that of Proposition \ref{t:tf} (Theorem 11.1 in \cite{BSS14}) although the statement is cast in a different language. Clearly one way to bound the transversal fluctuation is to show that any path with large transversal fluctuation has a much smaller weight than the weight of a typical polymer, this is the way the proof in \cite{BSS14} went. For the convenience of the reader we recall the main steps and indicate how the same argument can be used to prove Proposition \ref{p:tfnew}. For $j\geq 0$, let $C_{2,j}=\frac{C_2}{10^5}\prod_{i=0}^{j-1}(1+2^{-i/10})$. For $\ell=1,2, 3\ldots, 2^{j+1}-1$, let $\cB_{\ell,j}$ denote the event that there exists a path $\gamma$ from some $u$ in $A$ to some $v$ in $B$ with $\ell(\gamma)\geq 4r-C_1r^{1/3}$ and $\gamma$ intersects the line $x+y=\ell 2^{-j}r$ at some point $w=(w_1,w_2)$ with $|w_1-w_2|\geq C_{2,j}$. Let $\cB_{j}=\cup_{\ell} \cB_{\ell,j}$. Since the event in the proposition is in the event $\cup_{j=0}^\infty \cB_{j},$ it suffices to prove the following two lemmas. 
\begin{lemma}
\label{l:0}
In the above set-up we have $\P(\cB_0)\leq e^{-cC_2}$ if $C_2$ is a sufficiently large power of $C_1$. 
\end{lemma}

\begin{proof}
For convenience we divide the events $\cB_{0}$ into events $\cB_{0,k}$ depending on the location of the point $w_{*}$ as in the definition of the event, i.e., for $k\in \Z\setminus \{-1\}$, we say that $\cB_{0,k}$ holds if the event $\cB_{0}$ holds with a point $w_{*}$ located on the line segment $\sL_{k}$ joining $(r/2+ (C_2+kC)r^{2/3},r/2- (C_2+kC)r^{2/3})$ and $(r/2+ (C_2+(k+1)C)r^{2/3},r/2- (C_2+(k+1)C)r^{2/3})$. Clearly it suffices to prove that $\P(\cB_{0,k})\leq e^{-ck}e^{-cC_2}$. Clearly, 
$$\cB_{0,k} \subseteq \{\sup_{u\in A, w\in \sL_{k}, v\in B} T_{u,w}+T_{w,v} \geq 4r-C_1r^{1/3}\}.$$
We shall ignore the cases where $kr^{2/3}$ is linear in $r$, as such cases can be ruled out rather easily. For smaller values of $k$ observe that any line joining $u$ and $w$ and $w$ and $v$ must have slopes bounded away from $0$ and $\infty$, and hence Proposition \ref{t:treesup} applies. In particular this implies using Theorem \ref{t:moddev} as in the proof of Lemma \ref{l:p2l} that 
$$ \sup_{u\in A, w\in \sL_{k}} \E T_{u,w}= 2r+ O(C^2 r^{1/3})- O((C_2+kC)^2);$$
$$ \sup_{w\in \sL_{k}, v\in B} \E T_{w,v}= 2r+ O(C^2 r^{1/3})- O((C_2+kC)^2);$$
By choosing $C_2$ to be a large power of $C_1$ it follows that 
$$ \sup_{u\in A, w\in \sL_{k}, v\in B} \E T_{u,w}+ \E T_{w,v}\leq  4r-O(C_2^2+k^2)r^{1/3}.$$
Using Proposition \ref{t:treesup} we can now conclude $\P(\cB_{0,k})\leq e^{-ck}e^{-cC_2}$ as required.
\end{proof}

\begin{lemma}
\label{l:j}
In the above set up, for $j\geq 1$, and $1\leq \ell \leq 2^j$, we have $\P(\cB_{\ell,j}\cap \cB_{j-1}^{c})\leq 4^{-j}e^{-cC_2^{1/3}}$.
\end{lemma}

Before proving Lemma \ref{l:j} we need the following more general result.

\begin{lemma}
\label{l:general}
The following statement is true uniformly in all large $C$, $k$, and $r$. Suppose $\ch$ denote the event that the length of the best path from $A$ to $B$ that intersects the line $x+y=r$ outside $U_{C+k}$ is at least $4r-c_1k^{2}r^{1/3}$, where $c_1>0$ is some small constant. Then we have $\P(\ch)\leq C^2e^{-ck^2}$. 
\end{lemma}

\begin{proof}
Divide $A$ and $B$ into disjoint intervals $A_{i}, B_{i}, i=1,2,\ldots , C$ of length $r^{2/3}$ each. Now divide the event $\ch$ into $\ch_{i,j}$ where the endpoints of the paths are at $A_i$ and $B_j$ respectively. Clearly it suffices to prove that $\P(\ch_{i,j})\leq e^{-ck^2}$, this is done as in the proof of Lemma \ref{l:0} by noticing that expected length of any path from $A_i$ to the line $\{x+y=r\}$ outside $U_{C+k}$ and from there to $B_j$ is at most $4r-c_2k^2$ for some $c_2>0$. The proof is completed by an application of Proposition \ref{t:treesup} by taking $c_1$ smaller than $c_2$, and using a careful union bound as in Lemma \ref{l:0}. We omit the details. 
\end{proof}

We can now complete the proof of Lemma \ref{l:j}.

\begin{proof}[Proof of Lemma \ref{l:j}]
Observe first that it is only required to prove the lemma for odd $\ell$. Without loss of generality let us assume that $\ell=2m+1$. On $\cB^c_{j-1}$ we only need to control the length of the best path that crosses the line $\sL_{m}: x+y=m2^{-(j-1)}r$ and $\sL_{m+1}:=x+y=(m+1)2^{-(j-1)}r$ at points $u=(u_1, u_2)$ and $v=(v_1,v_2)$ respectively where $|u_1-u_2|, |v_1-v_2|\leq C_{2,j-1}r^{2/3}$. Observe that, for some fixed constant $K$, we have, using Lemma \ref{l:p2l}, with probability at least $1-4^{-j}e^{-cC_2^{1/3}}$ we have that the best path from $\mathbf{0}$ to $\sL_{m}$ has length at most $4m2^{-(j-1)}r+(C_2+K\log j)r^{1/3}$ and the best path from $\mathbf{r}$ to $\sL_{m+1}$ has length at most $4(r-(m+1)2^{-(j-1)r})+(C_2+K\log j)r^{1/3}$. Hence it suffices to prove that $$\P(\sup_{u,v} T^{*}_{u,v}\geq 4\times 2^{-(j-1)}r-(3C_2+2K\log j) r^{1/3})\leq 4^{-j}e^{-cC_2}$$ where the supremum is taken over all pairs of points $u,v$ described above and $T^*_{u,v}$ denotes the length of the best path from $u$ to $v$ that exits $U_{C_{2,j}}$ at the line $x+y=(2m+1)2^{-j}r$. This statement follows from Lemma \ref{l:general} by setting 
$$r=2^{-(j-1)}r, C=C_{2,j-1}2^{2(j-1)/3}, k=(C_{2,j}-C_{2,j-1})2^{2j/3}=C_{2,j-1}2^{-j/10}2^{2j/3}.$$ Observe that this provides for a penalty of $C^2_{2,j-1}2^{34j/30}2^{-j/3}r^{1/3}$ for the part of the path. We are done by observing (i)  this penalty is $\gg (C_2+K\log j) r^{1/3}$ for large $j$, and (ii) for our choices for $C$ and $k$ we have $C^2e^{-ck^2}\leq 4^{-j}e^{-cC_2^{1/3}}$, as required. 
\end{proof}

\end{document}